\documentclass[a4paper,11pt]{article}
\usepackage{amsmath,amssymb,amsthm}
\usepackage{mathrsfs}
\usepackage{color}
\usepackage{ascmac}
\usepackage{dsfont}
\usepackage{hyperref}
\theoremstyle{definition}
 \newtheorem{dfn}{Definition}[section]
 \newtheorem{remark}[dfn]{Remark}  
\theoremstyle{plain}
 \newtheorem{thm}[dfn]{Theorem}
 
 \newtheorem{lem}[dfn]{Lemma}

\numberwithin{equation}{section}
\newcommand{\sP}{\mathscr{P}}

\newcommand{\sS}{\mathscr{S}}
\newcommand{\sR}{\mathscr{R}}
\newcommand{\sL}{\mathscr{L}}

\newcommand{\sF}{\mathscr{F}}
\newcommand{\lr}[1]{\mathrm{L}_{#1}}
\newcommand{\rB}{{\mathrm{B}}}
\newcommand{\rH}{{\mathrm{H}}}
\newcommand{\rL}{{\mathrm{L}}}
\newcommand{\rW}{{\mathrm{W}}}

\newcommand{\rT}{{\mathrm{T}}}

\newcommand{\dv}{{\rm div}\,}

\newcommand{\BB}{{\mathbb B}}
\newcommand{\BR}{{\mathbb R}}
\newcommand{\BC}{{\mathbb C}} 

\newcommand{\BN}{{\mathbb N}}

\newcommand{\BT}{{\mathbb T}}
\newcommand{\torus}{{\mathbb T}}

\newcommand{\BZ}{{\mathbb Z}}

\newcommand{\CA}{{\mathcal A}}
\newcommand{\CB}{{\mathcal B}}

\newcommand{\CI}{{\mathcal I}}

\newcommand{\CL}{{\mathcal L}}

\newcommand{\CN}{{\mathcal N}}

\newcommand{\CR}{{\mathcal R}}

\newcommand{\CT}{{\mathcal T}}

\newcommand{\dd}{{\mathrm d}}
\newcommand{\e}{{\mathrm e}}

\newcommand{\bF}{{\mathbf F}}

\newcommand{\bG}{{\mathbf G}}

\newcommand{\bH}{{\mathbf H}}

\newcommand{\bh}{{\mathbf h}}

\newcommand{\bw}{{\mathbf w}}

\newcommand{\bv}{{\mathbf v}}
\newcommand{\bu}{{\mathbf u}}
\newcommand{\bff}{{\mathbf f}}
\newcommand{\bg}{{\mathbf g}}

\newcommand{\fp}{{\mathfrak p}}
\newcommand{\fq}{{\mathfrak q}}
\newcommand{\fr}{{\mathfrak r}}
\newcommand{\pd}{\partial}

\begin{document}
\title{Viscous flow past a translating body with oscillating boundary}

\author{Thomas Eiter\thanks{Weierstrass Institue for Applied Analysis and 
Stochastics, Mohrenstr. 39, 10117 Berlin, Germany.\endgraf
Email address: \texttt{thomas.eiter@wias-berlin.de}}
\and 
Yoshihiro Shibata\thanks{Department of Mathematics,  Waseda University, 
Ohkubo 3-4-1, Shinjuku-ku, Tokyo 169-8555, Japan. \endgraf
Email address: \texttt{yshibata325@gmail.com} \endgraf
Adjunct faculty member in the Department of 
Mechanical Engineering and Materials Science,
University of Pittsburgh, USA. \endgraf
Partially supported by Top Global University Project, 
JSPS Grant-in-aid for Scientific Research (A) 17H0109, and Toyota Central Research Institute Joint Research Fund.}}
\date{}
\maketitle

\begin{abstract}
We study an incompressible viscous flow around an obstacle 
with an oscillating boundary 
that moves by a translational periodic motion,
and we show existence of strong time-periodic solutions for small data in different configurations:
If the mean velocity of the body is zero,
existence of time-periodic solutions is provided within a framework of 
Sobolev functions with
isotropic pointwise decay. 
If the mean velocity is non-zero, 
this framework can be adapted,
but the spatial behavior of flow
requires a setting of anisotropically weighted spaces. 
In the latter case, we also establish existence of solutions within an
alternative framework
of homogeneous Sobolev spaces.
These results are based on the time-periodic maximal regularity 
of the associated linearizations,
which is derived 
from suitable $\sR$-bounds for the Stokes and Oseen resolvent problems.
The pointwise estimates are deduced from 
the associated time-periodic fundamental solutions.
\end{abstract}


\section{Introduction}

We consider a body with an oscillating boundary 
that moves through the three-dimensional space,
which is filled with an incompressible viscous fluid. 
The fluid motion is described by the Navier--Stokes equations
\begin{equation}\label{eq:system}
\pd_t\bu + \bu\cdot\nabla\bu - \mu\Delta\bu 
+\nabla\fp  = \bff, \quad \dv\bu = 0\quad\text{in 
$\Omega_t$}, \quad \
\bu|_{\Gamma_t}  = \bh,
\quad \
\lim_{|x|\to\infty}\bu(x,t)=0, 
\end{equation}
where $\Omega_t$ is the fluid domain with boundary $\Gamma_t$
at time $t\in\BR$.
The velocity field $\bu=(u_1,u_2,u_3)^\top$ 
and the pressure field $\fp$ are unknown,
while we prescribe the external force $\bff=(f_1,f_2,f_3)^\top$ 
and the fluid velocity at the boundary $\bh = (h_1, h_2, h_3)^\top$. 
The constant $\mu>0$ denotes the kinematic viscosity of the fluid.
Let $\CT>0$ be the time period of the boundary oscillation.
The body motion consists of a prescribed translation with velocity $\bv_\CB(t)\in\BR^3$
such that the fluid domain after one period is given by
\begin{equation}\label{eq:Omegaper}
\Omega_{t+\CT}=\Omega_t+\int_t^{t+\CT}\bv_\CB(\tau)\,\dd \tau
\end{equation}
for all times $t\in\BR$.
The translational velocity $\bv_\CB(t)$
is assumed to be time-periodic, that is,
$\bv_\CB(t+\CT)=\bv_\CB(t)$,
so that the displacement vector in \eqref{eq:Omegaper}
is independent of the time $t$. 
Moreover, by changing the frame of coordinates, 
we may assume that it is directed along the $x_1$-axis
such that the mean velocity over one period is given by
\[
\frac{1}{\CT}\int_t^{t+\CT}\bv_\CB(\tau)\,\dd \tau
=\frac{1}{\CT}\int_0^{\CT}\bv_\CB(\tau)\,\dd \tau
=\kappa e_1
\]
for some $\kappa\in\BR$,
where $e_1$ denotes unit vector in $x_1$-direction.
Then $\kappa$ corresponds to the mean translational speed
of the body.  
If the data $\bff$ and $\bh$ are time-periodic with period $\CT$,
then the whole system \eqref{eq:system} 
is time periodic in a frame moving with velocity $\kappa e_1$. 
In this article, we show existence of time-periodic solutions
to this problem
in suitable functional frameworks.
Observe that a natural choice of boundary conditions
in system~\eqref{eq:system}
would be classical no-slip conditions,
where $\bh$ is determined by the motion of the 
obstacle and its boundary.
In order to handle this choice, it is necessary to 
keep track of the dependence on the parameter $\kappa$
in the final existence result, 
see also Remark~\ref{rem:noslip} below.

A special case of the present situation is
when the body does not oscillate 
and its shape is time independent.
Then the problem reduces to the steady flow around a body
that translates with constant speed $\kappa$.
It is well known that the physical and mathematical properties
of this problem strongly depend on whether $\kappa=0$ or $\kappa\neq0$.
In particular, 
when $\kappa\neq0$, one observes a wake region behind the moving body,
which is reflected by an anisotropic decay of the velocity field.
As was shown recently, the same behavior can be observed 
for the time-periodic flow past a rigid body \cite{Eiter_Asympt}.
This observation suggests the necessity 
to also distinguish the cases $\kappa=0$ and $\kappa\neq 0$ 
in the presence of an oscillating boundary.

The mathematically rigorous study of
time-periodic Navier--Stokes flows
was initiated in the works of
Serrin~\cite{Serrin_PeriodicSolutionsNS1959},
Prodi~\cite{Prodi1960},
Yudovich~\cite{Yudovich1960}
and Prouse~\cite{Prouse1964}.
While these articles focus on bounded domains,
the first analytical result on time-periodic solutions to the 
Navier--Stokes equations
in an unbounded domain was achieved 
decades later by Maremonti~\cite{Maremonti_TimePer91},
who derived existence of time-periodic solutions 
in the three-dimensional whole space within an $\rL^2$ framework.
In the case of the presence of an exterior domain,
the first results on existence of time-periodic solutions 
are due to Salvi~\cite{Salvi1995} and Maremonti and Padula~\cite{MaremontiPadula1999}
for $\kappa=0$.
The case $\kappa\neq 0$ is 
due to Galdi and Silvestre~\cite{GaldiSilvestre2006},
who considered the situation of a general time-periodic rigid motion.
Existence of time-periodic mild solutions in so-called weak Lebesgue spaces 
is due to the fundamental work by Yamazaki~\cite{Yamazaki2000} for $\kappa=0$,
which was further developed to a general approach to time-periodic problems
by Geissert, Hieber and Nguyen~\cite{GeissertHieberNguyen16},
who also treated the case $\kappa\neq0$.
However, the classes of solutions studied in these articles
do not give suitable information on the decay of the flow
far from the body. 
This issue was addressed by Galdi and Sohr~\cite{GaldiSohr2004} 
for $\kappa=0$,
and very recently by Galdi~\cite{Galdi21} for $\kappa\neq 0$,
who established existence of regular solutions with pointwise spatial decay.
The asymptotic behavior is also reflected in
the framework of homogeneous Sobolev spaces 
introduced by  
Galdi and Kyed \cite{GaldiKyed_TPFlow2018},
who showed existence of 
time-periodic strong solutions in the case $\kappa\neq 0$
based on a framework 
of time-periodic maximal $\rL^p$ regularity
for the corresponding Oseen linearization.
As was shown recently by
Eiter, Kyed and Shibata~\cite{EiterKyedShibata_PeriodicLp}
in a more general framework,
a combination of this approach with 
suitable pointwise estimates leads to the existence of time-periodic solutions
for $\kappa=0$. 

Nearly all of the previous articles 
are concerned with Navier--Stokes flows in a domain with a fixed boundary.
Only in~\cite{Salvi1995}, 
the flow in an exterior domain with periodically moving boundary
was considered, 
and existence of time-periodic weak solutions
was shown.
While the corresponding problem in a bounded domain
has been addressed by several researchers in a framework of
weak solutions~\cite{Morimoto71,MiyakawaTeramoto82, Salvi94},
time-periodic mild solutions in a bounded domain were recently established
by Farwig, Kozono, Tsuda and Wegmann~\cite{FKTW21}
via a semigroup approach.
These solutions were later shown to be strong~\cite{FarwigTsuda_FujitaKato22}.
Independently, 
Eiter, Kyed and Shibata~\cite{EiterKyedShibata_PeriodicLp}
derived existence of strong solutions in a bounded domain
from the aforementioned framework of time-periodic maximal regularity
without relying on semigroup theory.
This approach was also used to establish time-periodic solutions 
in the case of one-phase and two-phase
flows~\cite{EiterKyedShibata_PeriodicOneTwoPhase21}.
In the present article, we follow this strategy
to establish first results 
on the existence of time-periodic strong solutions to the
Navier--Stokes equations in an exterior domain with an oscillating boundary.

We begin by transforming~\eqref{eq:system}
to a problem in a 
time-independent reference domain $\Omega$ with boundary $\Gamma$.
A suitable linearization leads to the system
\begin{equation}\label{eq:lin.intro}
\pd_t\bv - \mu\Delta\bv - \kappa\partial_1\bv + \nabla\fp 
= \bff,
\quad \dv\bv = 0\quad 
\text{in $\Omega\times\BT$}, \qquad 
\bv|_{\Gamma\times\BT}  = \bh.
\end{equation}
Here $\BT=\BR/\CT\BZ$ denotes the torus group 
associated with the given time period $\CT$,
and it indicates that all functions occurring in~\eqref{eq:lin.intro}
are time periodic.
To treat the full nonlinear problem by a fixed-point argument,
we first derive a result on maximal regularity 
for this time-periodic linear problem,
that is,
the existence of unique solutions to \eqref{eq:lin.intro}
that satisfy an \textit{a priori} estimate of the form
\[
\begin{aligned}
\|\pd_t\bv_\perp\|_{\rL_p(\BT, \rL_q(\Omega))}
+ \|\nabla^2\bv_\perp\|_{\rL_p(\BT, \rL_q(\Omega))}
&+|\kappa|\,\|\partial_1\bv_S\|_{\rL_p(\BT, \rL_q(\Omega))}
+ \|\nabla\fp\|_{\rL_p(\BT, \rL_q(\Omega))}
\\
&\qquad\qquad\quad
\leq C\big(
\|\bff\|_{\rL_p(\BT, \rL_q(\Omega))}
+\|\bh\|_{\rT_{p,q}(\Gamma\times\BT)}
\big)
\end{aligned}
\]
for suitable $p,q\in(1,\infty)$, 
where $\rT_{p,q}(\Gamma\times\BT)$ is a suitable trace space
introduced below.

To derive time-periodic maximal regularity, 
following the approach 
from~\cite{EiterKyedShibata_PeriodicOneTwoPhase21,EiterKyedShibata_PeriodicLp},
we first investigate the associated resolvent problems.
Observe that
the Fourier coefficients 
$(\hat\bv_k,\hat\fp_k)=(\sF_\BT[\bv](k),\sF_\BT[\fp](k))$, $k\in\BZ$,
of a time-periodic solution $(\bv,\fp)$ to~\eqref{eq:lin.intro}
satisfy
\begin{equation}\label{eq:res.intro}
i\frac{2\pi}{\CT}\hat\bv_k - \mu\Delta\hat\bv_k - \kappa\partial_1\hat\bv_k + \nabla\hat\fp_k 
= \hat\bff_k,
\qquad \dv\hat\bv_k = 0\quad 
\text{in $\Omega$}, \qquad 
\hat\bv_k|_{\Gamma}  = \hat\bh_k,
\end{equation}
and if $\CA(k)$ is a solution operator for \eqref{eq:res.intro}
such that $(\hat\bv_k,\hat\fp_k)=\CA(k)(\hat\bff_k,\hat\bh_k)$,
then a solution $(\bv,\fp)$ to~\eqref{eq:lin.intro}
is formally given by
\begin{equation}
\label{eq:Fmult}
(\bv,\fp)=\sF_\BT^{-1}\big[
k\mapsto\CA(k)\sF_\BT[(\bff,\bh)](k)
\big].
\end{equation}
This formula reduces the question of maximal regularity 
to the investigation of an operator-valued Fourier multiplier,
for which we apply an abstract multiplier theorem 
based one the notion of $\sR$-boundedness.
This application is complicated by the fact
that the required $\sR$-bounds for the family $\{\CA(k)\}$ 
are only available for large $k$.
Therefore, we proceed as in~\cite{EiterKyedShibata_PeriodicLp}
and decompose the time-periodic solution into a high-frequency part
and a low-frequency part.
While the first can then be treated by means of Fourier multipliers on $\BT$,
the latter consists of finitely many Fourier modes
that can be handled separately.

As demonstrated in~\cite{EiterKyedShibata_PeriodicLp},
the resulting framework of maximal regularity is suitable to treat 
the nonlinear problem~\eqref{eq:system} in the case of a bounded domain.
Compared to this, the present setting of an 
exterior domain comes along with two difficulties:
Firstly, for $k=0$ 
the resolvent problem~\eqref{eq:Fmult}
is not uniquely solvable in a framework of classical Sobolev spaces.
Therefore,
we decompose
the time-periodic linearized problem~\eqref{eq:lin.intro}
into the associated steady-state problem
and a purely oscillatory problem,
which are studied in separate functional frameworks.
Secondly, the treatment of the nonlinear problem
requires suitable estimates of the nonlinear terms.
Those cannot be derived within the resulting 
maximal-regularity function class, at least for $\kappa=0$.
Therefore, we complement the setting 
with pointwise estimates
that are derived from the time-periodic fundamental solutions to~\eqref{eq:lin.intro},
which were introduced by Eiter and Kyed~\cite{EK1}.
In the resulting framework, the contraction mapping principle can be 
used to derive existence of time-periodic solutions to the full nonlinear problem
for $\kappa=0$, see Theorem~\ref{mainthm:Stokes}.
In the case $\kappa\neq 0$, 
the same method can be employed, which results in Theorem~\ref{mainthm:Oseen},
but we have to take into account the anisotropic decay of the flow.
However, since the maximal-regularity framework for $\kappa\neq 0$
leads to better integrability properties of solutions,
we can also implement a fixed-point argument without
enriching the functional setting with pointwise decay properties,
see Theorem~\ref{mainthm:Oseen.int}.

We want to emphasize 
that our analysis also leads to new existence results in the case
of a moving body with a fixed boundary. 
In~\cite{GaldiKyed_TPFlow2018}, 
existence of strong time-periodic solutions was established for $\kappa\neq 0$
under the assumption that 
the translational velocity has the same direction at all times,
that is, $\bv_\CB(t)=v_\CB(t) e_1$ for some scalar time-periodic function $v_\CB$
with non-zero mean.
This assumption was necessary
to work with an Oseen linearization in a frame attached to the body,
that is, moving with the actual translational velocity $\bv_\CB$.
In contrast, 
we shall work in a frame moving with the \emph{mean} translational velocity $\kappa e_1$.
Regarding the remainder $\bv_\CB-\kappa e_1$ as an oscillation of the boundary,
we can omit the restrictions on $\bv_\CB$ from~\cite{GaldiKyed_TPFlow2018}.
For more details, see Section~\ref{sec:reformulation}.

The structure of the article is as follows:
After introducing the general notation in Section~\ref{sec:notation},
we reformulate the nonlinear system~\eqref{eq:system}
as a problem on a fixed domain in Section~\ref{sec:reformulation}.
In Section~\ref{sec:mainresults}, 
we then state our main results on the nonlinear and the linear problems.
The proofs for the linear theory for $\kappa\neq 0$
are provided in Section~\ref{sec:Oseen},
while they were provided in~\cite{EiterKyedShibata_PeriodicLp} for $\kappa=0$.
In Section~\ref{sec:nonlin}
we conclude by the proofs of the existence results
for the nonlinear problem~\eqref{eq:system}.

\section{Notation}
\label{sec:notation}

For topological vector spaces $X$ and $Y$,
we denote the space of continuous linear operators from
$X$ to $Y$
by $\sL(X, Y)$. 
We write $X'$ for the dual space of $X$,
and when $X$ is a normed space, then $\|\cdot\|_X$ denotes is norm.

By $\Omega$ we denote a three-dimensional exterior $C^2$-domain,
that is, a domain that is the complement of a compact set
with connected $C^2$-boundary $\Gamma=\partial\Omega$.
Let $b > 0$ be a sufficiently large radius such 
that $\Gamma\subset B_b := \{x \in \BR^3 \mid |x| < b\}$.

We write $\partial_j:=\partial_{x_j}$ for partial derivatives in space,
and $\nabla$, $\dv$ and $\Delta$  
denote gradient, divergence and Laplace operator,
which only act in spatial variables. 
For a sufficiently regular function $u$ and $k\in\BN$,
we denote the collection of all $k$-th order derivatives by $\nabla^k u$.

For classical Lebesgue and Sobolev spaces 
we write $\rL_q(\Omega)$ and $\rH^k_q(\Omega)$,
where $q\in[1,\infty]$ and $k\in\BN$,
and $\rL_{q,\mathrm{loc}}(\Omega)$ and $\rH^k_{q,\mathrm{loc}}(\Omega)$
denote their local variants.
Homogeneous Sobolev spaces are defined via
\[
\hat\rH^k_q(\Omega):=\big\{
u\in\rL_{1,\mathrm{loc}}(\Omega) \mid \nabla^k u\in \rL_q(\Omega)
\big\}
\]
When it is clear from the context,
we sometimes use the same notation for spaces of
vector-valued or matrix-valued functions.
For example, we write $\rL_q(\Omega)$ instead of 
$\rL_q(\Omega)^3$ or $\rL_q(\Omega)^{3\times3}$.

For a given time period $\CT>0$
we let $\BT:= \BR/\CT\BZ$ denote the corresponding torus group.
Functions on $\BT$ can then be identified with
$\CT$-periodic functions on $\BR$,
which we do tacitly in what follows.
We equip the topological group $\BT$ with the normalized Haar measure defined via
\[
\forall u \in C^0(\BT):\quad \int_\BT u(t)\,\dd t 
= \frac{1}{\CT}\int_0^{\CT} u(\tau)\,\dd \tau.
\]
Bochner--Lebesgue spaces
are denoted by $\rL_p(\BT;X)$ for $p\in[1,\infty]$, 
and we set
\[
\begin{aligned}
&\rH^{1}_{p}(\torus, X)
:= \{u \in\lr{p}(\torus, X) \mid \partial_t u
 \in\lr{p}(\torus, X)\},
 \quad
\|u\|_{\rH^1_p(\torus, X)} :=
 \|u\|_{\lr{p}(\torus, X)}+\|\partial_t u\|_{\lr{p}(\torus, X)}.
\end{aligned}
\]
The velocity field $\bu$ of a solution 
will be identified in the classical parabolic space, at least near the boundary,
that is,
$\bu\in\rH^1_p(\BT, \rL_q(\Omega_b)^3) \cap 
\rL_p(\BT, \rH^2_q(\Omega_b)^3)$
for some $p,q\in(1,\infty)$.
To shorten notation, 
we introduce the corresponding class of boundary traces on $\Gamma\times\BT$
via
\begin{equation}
\label{eq:tracespace}
\rT_{p,q}(\Gamma\times\BT)
:=\big\{ \bh=\bu|_{\Gamma\times\BT} \ \big | \ \bu\in\rH^1_p(\BT, \rL_q(\Omega)^3) \cap 
\rL_p(\BT, \rH^2_q(\Omega)^3) \big\}
\end{equation}
with corresponding norm
\[
\|\bh\|_{\rT_{p,q}(\Gamma\times\BT)}
:=\inf \big\{ 
\|\bu\|_{\lr{p}(\torus, \rH^2_q(\Omega))} 
+ \|\partial_t \bu\|_{\lr{p}(\torus, \rL_q(\Omega))} 
\ \big| \
\bh=\bu|_{\Gamma\times\BT}\big\}.
\]
Note that $\rT_{p,q}(\Gamma\times\BT)$ 
can be identified with a real interpolation space,
but we shall not make use of this property in what follows.

We often decompose
time-periodic functions $f\colon\Omega\times\BT\to\BR$
into a steady-state part $f_S$ 
and a purely oscillatory part $f_\perp$
defined by 
\begin{equation}\label{eq:decomposition}
f_S(x)= \int_\BT f(x, t)\,\dd t,
\qquad
f_\perp(x, t) = f(x, t) - f_S(x).
\end{equation}
To quantify decay rates of these two parts,
we use the norms
\[
<f_S>_\alpha = \sup_{x \in \Omega} |f_S(x)|
(1 + |x|)^{\alpha}, \qquad 
<f_S>^w_{\alpha,\beta}  = \sup_{x \in \Omega} |f_S(x)|(1+|x|)^\alpha(1+|x|-x_1)^\beta
\]
for the steady-state parts
and 
\[
<f_\perp>_{p, \alpha}  = \sup_{x \in \Omega} \|f_\perp(x, \cdot)\|_{\rL_p(\BT)}
(1+|x|)^{\alpha}
\]
for the purely oscillatory part,
where $\alpha,\beta\geq 0$ and $p\in(1,\infty)$.

We equip $\Omega\times\BT$ with the product measure
and denote the associated Lebesgue spaces by $\rL_p(\Omega\times\BT)$.
If $\Omega=\BR^3$, then $\BR^3\times\BT$ 
is a locally compact abelian group,
and we can define generalized Schwartz spaces 
and spaces of tempered distributions
on $\BT$ and $\BR^3\times\BT$ 
and the respective dual groups $\BZ$ and $\BZ\times\BR^3$.
see~\cite{Bruhat61,EiterKyed_tplinNS_PiFbook}.
Moreover, there is an associated notion of Fourier transform,
which can be defined via
\[
\sF_\BT[u](k) := \int_\BT u(t)\,\e^{-i\frac{2\pi}{\CT}kt}\,\dd t,
\qquad
\sF_\BT^{-1}[w](t) := \sum_{k \in \BZ} w(k)\,\e^{i\frac{2\pi}{\CT}kt},
\]
and we set
$\sF_{\BR^3\times\BT}=\sF_{\BR^3}\otimes\sF_{\BT}$ 
and 
$\sF_{\BR^3\times\BT}^{-1}=\sF_{\BR^3}^{-1}\otimes\sF_{\BT}^{-1}$,
where the Fourier transform in the Euclidean setting and its inverse 
are given by
\[
\sF_{\BR^3}[u](\xi):= \frac{1}{(2\pi)^3}\int_{\BR^3}u(x)\,\e^{-i\xi\cdot x}\,\dd x, 
\qquad
\sF_{\BR^3}^{-1}[w](x):= \int_{\BR^3}w(\xi)\,\e^{i\xi\cdot x}\,\dd \xi.
\]

To study operator-valued Fourier multipliers, we further need the notion of
UMD spaces,
which are Banach spaces $X$ such that 
the Hilbert transform $H$, defined by
\[
Hf(t) := \frac{1}{\pi}\lim_{\varepsilon\to0}\int_{|x|\geq\varepsilon} \frac{f(t-s)}{s}\,\dd s,
\]
is a bounded linear operator on $\rL_{p}(\BR,X)$ some $p\in(1,\infty)$. 
Moreover, we say that a family of operators $\CT \subset \sL(X, Y)$ is 
$\sR$-bounded in $\sL(X, Y)$
if there is $C > 0$
such that
\begin{equation}\label{est:Rbound}
\bigl\|\sum_{k=1}^n r_kT_kf_k\bigr\|_{\lr{1}((0,1), Y)} \leq 
C\bigl\|\sum_{k=1}^nr_kf_k\bigr\|_{\lr{1}((0, 1), X)}
\end{equation}
for all
$n \in \BN$, $\{T_j\}_{j=1}^n \in \CT^n$, 
and $\{f_j\}_{j=1}^n \in X^n$.
Here $r_k\colon [0, 1] \to \{-1, 1\}$, $t \mapsto {\rm sign}\,(\sin 2^k\pi t)$,
denote Rademacher functions.
Moreover, $\sR_{\sL(X, Y)}\CT$ denotes
the smallest constant $C$ such that \eqref{est:Rbound} holds.

For $\varepsilon\in(0,\pi/2)$ and $\delta>0$
we define the perturbed sector
\[
\Sigma_{\varepsilon,\delta}
:=\{\lambda\in\BC \mid |\lambda|>\delta, \, |\arg \lambda| <\pi-\varepsilon\}.
\]
Moreover, 
${\rm Hol}\,(\Sigma_{\varepsilon, \delta}, X)$
denotes the class of $X$-valued holomorphic functions on $\Sigma_{\varepsilon,\delta}$.

\section{Formulation on a reference domain}
\label{sec:reformulation}

To reformulate the 
system~\eqref{eq:system}
as a problem in a time-independent spatial domain, 
we describe the motion of the body and its boundary by suitable functions.

Let $\Omega$ be the exterior domain in $\BR^3$.
Let $\bv_\CB\in C^0(\BR)^3$
with $\bv_\CB(t+\CT)=\bv_\CB(t)$ for all $t\in\BR$,
and let
\begin{equation}
\label{eq:trafo.regularity}
\phi\in
C^0(\BR;C^3(\Omega)^3)\cap
C^1(\BR;C^1(\Omega)^3)
\end{equation}
such that $\phi(y, 0) = 0$ and $\phi(y, t+\CT) = \phi(y, t)$ for 
each $t \in \BR$ and $y \in \Omega$,
and such that $\phi(y, t) = 0$ for $y \not\in B_{2b}$.
Then the fluid domain $\Omega_t\subset\BR^3$ shall be given 
by 
\begin{equation}
\Omega_t = \Big\{x = y + \phi(y, t) + \int_0^t\bv_\CB(\tau)\,\dd \tau \,\Big|\, y \in \Omega\Big\} \quad(t \in \BR).
\label{eq:Omegat.direct}
\end{equation}
By rotating the coordinate frame, 
we may assume that the mean velocity over one time period
is directed along the $x_1$-axis
such that there is $\kappa\in\BR$ with
\[
\kappa e_1
=\frac{1}{\CT}\int_0^{\CT}\bv_\CB(\tau)\,\dd \tau.
\]
Then we can redefine $\phi$ in such a way that
\begin{equation}
\Omega_t = \Big\{x = y + \phi(y, t) + t\kappa e_1 \,\Big|\, y \in \Omega\Big\} \quad(t \in \BR)
\label{eq:Omegat.simple}
\end{equation}
instead of \eqref{eq:Omegat.direct}.
Indeed,
since
\[
\phi(y, t) + \int_0^t\bv_\CB(\tau)\,\dd \tau
= \Big(\phi(y, t) + \int_0^t(\bv_\CB(\tau)-\kappa e_1)\,\dd \tau \Big )
+ t\kappa e_1,
\]
we may assume that $\bv_\CB$ is constant in time 
and replace $\phi$ with the term in parenthesis,
which defines a time-periodic function.
Notice that to preserve the condition
$\phi(y, t) = 0$ for $|y|>2b$,
it might be necessary to multiply the term by a suitable cut-off function,
which would not change the set $\Omega_t$.

Given $\kappa$ and $\phi$, 
the domain $\Omega_t$ is the image of the transformation
$\Phi_t\colon\Omega\to\BR^3$,
$\Phi_t(y)=y+\phi(y,t)+t\kappa e_1$,
for $t\in\BR$,
and the boundary $\Gamma_t=\partial\Omega_t$ is given by
$\Gamma_t = \{x = y + \phi(y, t)+t\kappa e_1 \mid y \in \Gamma\}$.
To reduce system~\eqref{eq:system}
to a problem in the reference domain $\Omega=\Omega_0$,
we assume that
\begin{equation}\label{6.2} 
\sup_{t \in \BR} \|\phi(\cdot, t)\|_{\rH^3_\infty(\Omega)}
+ \sup_{t \in \BR} \|\pd_t\phi(\cdot, t)\|_{\rH^1_\infty(\Omega)} 
\leq \varepsilon_0
\end{equation}
with some small number $\varepsilon_0>0$,
and we use the change of variables induced by $\Phi_t$,
namely $x = y + \phi(y, t)+t\kappa e_1$.
By the smallness assumption~\eqref{6.2}, we may assume the existence of the
inverse transformation,
which has the form $y = x + \psi(x, t)-t\kappa e_1$.
The associated Jacobi matrix
$\pd(t, y)/\pd(t, x)$ is given by the formulas:
\[
\frac{\pd t}{\pd t} =1, \quad  \frac{\pd t}{\pd x_j} = 0, \quad
\frac{\pd y_\ell}{\pd t}=\frac{\pd \psi_\ell}{\pd t}-\kappa e_1, 
\quad \frac{\pd y_\ell}{\pd x_j} = \delta_{\ell j} + \frac{\pd \psi_\ell}{\pd x_j}
\]
for $j, \ell=1,2,3$.  Set 
\[
a_{\ell0}(y, t) =(\pd\psi_\ell/\pd t)(y+\phi(y, t)+t\kappa e_1, t),
\qquad
a_{\ell j}(y, t) = (\pd\psi_\ell/\pd x_j)(y + \phi(y, t)+t\kappa e_1, t).
\] 
Then partial derivatives transform as
\begin{equation}\label{5.3}
\frac{\pd f}{\pd t} = \frac{\pd g}{\pd t} + \sum_{\ell=1}^3 a_{\ell0}(y, t)\frac{\pd g}{\pd y_\ell}
-\kappa\frac{\partial g}{\partial y_1},
\qquad 
\frac{\pd f}{\pd x_j} = \frac{\pd g}{\pd y_j} + \sum_{\ell=1}^3 a_{\ell j}(y, t)\frac{\pd g}{\pd y_\ell}
\end{equation}
for $f(x,t)=g(y,t)$.
Let ${\rm J} = \det(\pd x/\pd y) = 1 + {\rm J}_0(y, t)$ be the Jacobian of $\Phi_t$.
From \eqref{6.2}
we obtain $C>0$ such that
\begin{equation}\label{5.4}\begin{aligned}
&\sup_{t \in \BR}\|a_{\ell j}(\cdot, t)\|_{\rH^2_\infty(\Omega)}
+ \sup_{t \in \BR}\|\pd_t a_{\ell j}(\cdot, t)\|_{\rL_\infty(\Omega)}
+ \sup_{t \in \BR}\|a_{0j}(\cdot, t)\|_{\rL_\infty(\Omega)} \\
&\qquad +
\sup_{t\in\BR}\|{\rm J}_0(\cdot, t)\|_{\rH^2_\infty(\Omega)} 
+ \sup_{t\in\BR}\|\pd_t{\rm J}_0(\cdot, t)\|_{\rL_\infty(\Omega)} \leq C\varepsilon_0
\end{aligned}\end{equation}
for $j, \ell=1,2,3$.  
For $\bv(y, t) 
=(v_1, v_2, v_3)^\top= \bu(x, t)$, and $\fq(y, t)
= \fp(x, t)$
we then have 
\[
\begin{aligned}
&\pd_t\bu = \pd_t\bv + \sum_{\ell=1}^3 a_{\ell0}\frac{\pd\bv}{\pd y_\ell}
-\kappa\frac{\partial\bv}{\partial y_1}, 
\quad 
\bu\cdot\nabla\bu = \bv\cdot({\rm I} + {\rm A})\nabla\bv,
\\ 
&\Delta\bu = \Delta\bv + \sum_{\ell=1}^3(a_{\ell j}+ a_{j\ell})\frac{\pd^2\bv}{\pd y_\ell \pd y_j}
+ \!\!\sum_{j,\ell, m=1}^3 a_{\ell j}a_{mj}\frac{\pd^2\bv}{\pd y_\ell\pd y_m} 
+ \sum_{\ell, m=1}^3\left(\frac{\pd a_{m\ell}}{\pd y_\ell} 
+ \sum_{j=1}^3 a_{\ell j}\frac{\pd a_{mj}}{\pd y_\ell}\right)\frac{\pd \bv}{\pd y_m},   
\\
&\dv\bu 
= {\rm J}^{-1}\Big(\dv \bv + \dv({\rm J}_0\bv) 
+ \sum_{j,\ell=1}^3\frac{\pd}{\pd y_\ell}(a_{\ell j}{\rm J}v_j)\Big), \quad
\nabla\fp = ({\rm I} + {\rm A})\nabla\fq,
\end{aligned}
\]
where ${\rm A}$ is a $(3\times 3)$-matrix whose $(j,k)$-th component is
$a_{jk}$.  Setting $w_\ell = v_\ell + {\rm J}_0v_\ell + \sum_{j=1}^3 a_{\ell j}{\rm J}v_j$, 
we have ${\rm J}\dv\bu = \dv\bw$ with $\bw = (w_1, w_2, w_3)^\top$.  Notice that 
$\bw = ({\rm I} + {\rm J}_0{\rm I} + {\rm A}^\top{\rm J})\bv$.
In view of \eqref{5.4}, choosing $\varepsilon_0 > 0$ sufficiently small, we see that 
there exists a $(3\times 3)$-matrix ${\rm B}_{-1}$ such that 
$({\rm I} + {\rm J}_0{\rm I} + {\rm A}^\top{\rm J})^{-1} = {\rm I}+ {\rm B}_{-1}$
and 
\begin{equation}\label{5.4*}
\sup_{t \in \BR} \|{\rm B}_{-1}(\cdot, t)\|_{\rH^2_\infty(\Omega)} \leq C\varepsilon_0,
\quad \sup_{t \in \BR} \|\pd_t{\rm B}_{-1}(\cdot, t)\|_{\rL_\infty(\Omega)}
\leq C\varepsilon_0.
\end{equation}
We further replace the time axis with the torus group 
$\BT:=\BR/\CT\BZ$ associated to the period $\CT$.
In total, system \eqref{eq:system} is transformed to
\begin{equation}\label{eq:systemref}
\pd_t\bw - \mu\Delta\bw - \kappa \pd_1\bw + 
\nabla\fq = \bff + \CL(\bw, \fq) + \CN(\bw),
\quad \dv \bw = 0\quad
\text{in $\Omega\times\BT$}, \qquad 
\bw|_{\Gamma\times\BT}  = \bh,
\end{equation}
where the data $\bff$ and $\bh$ are now prescribed 
with respect to the reference domain $\Omega$,
and 
\begin{equation}\label{6.6}\begin{aligned}
\CL(\bw, \fq)  &= -\pd_t({\rm B}_{-1}\bw)  
-\sum_{\ell=1}^3 a_{\ell0}\frac{\pd}{\pd y_\ell}(({\rm I}+{\rm B}_{-1})\bw)
+ \mu\Delta({\rm B}_{-1}\bw) \\
&\quad
+ \sum_{\ell=1}^3(a_{\ell j}+ a_{\ell j})
\frac{\pd^2}{\pd y_\ell \pd y_j}(({\rm I} + {\rm B}_{-1})\bw)
+ \sum_{j,\ell, m=1}^3 a_{\ell j}a_{mj}
\frac{\pd^2}{\pd y_\ell\pd y_m}(({\rm I} + {\rm B}_{-1})\bw) \\
&\quad + \sum_{\ell, m=1}^3\left(\frac{\pd a_{m\ell}}{\pd y_\ell} 
+ \sum_{j=1}^3 a_{\ell j}\frac{\pd a_{mj}}{\pd y_\ell}\right)
\frac{\pd }{\pd y_m}(({\rm I} + {\rm B}_{-1})\bw)\\
&\quad -\kappa\{\pd_1(\rB_{-1}\bw)+\sum_{\ell=1}^3a_{1\ell}
\frac{\pd}{\pd y_\ell}(({\rm I} + {\rm B}_{-1})\bw\}
- A\nabla\fq,
\\
\CN(\bw) & = (({\rm I} + {\rm B}_{-1})\bw)\cdot({\rm I} 
+ {\rm A})\nabla(({\rm I} + {\rm B}_{-1})\bw).
\end{aligned}\end{equation}
Notice that in~\eqref{eq:systemref}
we omitted that $\bw$ vanishes at infinity.
This condition will later be included in a suitable sense
in the definition of the function spaces.

\begin{remark}
\label{rem:noslip.formula}
In both formulations~\eqref{eq:system} and \eqref{eq:systemref},
we consider a general class of boundary data 
in the from of an inhomogeneous Dirichlet condition.
The most classical choice would be given by no-slip conditions
such that the fluid velocity coincides with the boundary velocity.
With the notation from above,
this means to assume that
\[
\bu(x,t)=\bv(y,t)=\partial_t\Phi_t(y)=\partial_t\phi(t,y)+\kappa e_1
\]
for $y\in\Gamma$ and $x=\Phi_t(y)=y+\phi(y,t)+t\kappa e_1\in\Gamma_t$.
Therefore, no-slip conditions 
correspond to the choice
\begin{equation}
\bh=({\rm I} + {\rm J}_0{\rm I} + {\rm A}^\top{\rm J})(\partial_t\phi+\kappa e_1)
\label{eq:noslipdata}
\end{equation}
in system~\eqref{eq:systemref}.
In particular, 
the prescribed boundary data $\bh$
depend on the translational velocity $\kappa$,
which also appears as a parameter in the linearization of~\eqref{eq:systemref}.
Therefore, for the treatment of no-slip conditions,
the dependence of smallness conditions on $\kappa$
has to be taken into account. 
In Remark~\ref{rem:noslip} 
we clarify in how far no-slip conditions can be handled 
in the frameworks proposed here.
\end{remark}

\section{Main results}
\label{sec:mainresults}

We first state the results on 
existence of time-periodic solutions
to~\eqref{eq:systemref}.
Their proofs will be based on the study 
of a suitable linearization,
which is given by
\begin{equation}\label{eq:nslin.tp}
\pd_t\bv - \mu\Delta\bv - \kappa\partial_1\bv + \nabla\fp 
= \bff,
\quad \dv\bv = 0\quad 
\text{in $\Omega\times\BT$}, \qquad 
\bv|_{\Gamma\times\BT}  = \bh.
\end{equation}
We obtain a time-periodic Stokes system for $\kappa=0$,
and a time-periodic Oseen problem for $\kappa\neq 0$,
which have different mathematical properties.  
The results on unique existence of solutions to
the linear problem~\eqref{eq:nslin.tp}
are collected in Subsection~\ref{subsec:linresults}.

\subsection{Solutions to the nonlinear problem}

In the theorems on existence of solutions to problem~\eqref{eq:systemref},
we always assume
\begin{equation}
\label{eq:assumptions}
2<p<\infty, \quad
3<q<\infty, \quad
h\in \rT_{p,q}(\Gamma\times\BT),
\quad
\phi\in
C^0(\BR;C^3(\Omega)^3)\cap
C^1(\BR;C^1(\Omega)^3),
\end{equation}
where $\rT_{p,q}(\Gamma\times\BT)$ 
is the space from \eqref{eq:tracespace}.
We begin with the case without translation, that is, where
$\kappa=0$ in \eqref{eq:systemref}. 
To quantify the pointwise decay of functions, 
we use the weighted norms introduced in Section~\ref{sec:notation}.

\begin{thm}\label{mainthm:Stokes} 
Assume~\eqref{eq:assumptions}
and let $\bff = \bff_S + \bff_\perp$
with $\bff_S = \dv\bF_S$ and $\bff_\perp = \dv \bF_\perp$.
There exist constants $\varepsilon,\varepsilon_0>0$ such that if 
the smallness conditions \eqref{6.2} and
\begin{equation}\label{6:small.2}\begin{aligned}
<\bff_S>_3 + <\bF_S>_2 + <\bff_\perp>_{p, 2} + <\bF_\perp>_{p, 1} 
+\|\bh\|_{\rT_{p,q}(\Gamma\times\BT)} &< \varepsilon^2
\end{aligned}\end{equation}
are satisfied,
then problem \eqref{eq:systemref} with $\kappa=0$  admits
a unique solution $(\bw,\fp)$ with
$$
\bw \in \rH^1_p(\BT, \rL_q(\Omega)^3) \cap \rL_p(\BT, \rH^2_q(\Omega)^3),  \quad
\fp \in \rL_p(\BT, \hat\rH^1_q(\Omega))
$$
satisfying the estimate
\begin{align*}
<\bw>_{p,1} + <\nabla\bw>_{p,2}  
+ \|\bw\|_{\rL_p(\BT, \rH^2_q(\Omega))}
+ \|\pd_t\bw\|_{\rL_p(\BT, \rL_q(\Omega))}
+ \|\nabla\fp\|_{\rL_p(\BT, \rL_q(\Omega))}
\leq \varepsilon.
\end{align*}
\end{thm}

In the case with translation, that is, where $\kappa \neq 0$ in 
\eqref{eq:systemref},
we obtain existence of a time-periodic solution
with anisotropic pointwise decay. 

\begin{thm}\label{mainthm:Oseen}
Let $\kappa_0>0$ and $\delta \in (0, 1/4)$,
and assume~\eqref{eq:assumptions}.  
Let $\bff= \bff_S + \bff_\perp$ with $\bff_\perp=\dv \bF_\perp$.
Then there exist
$\varepsilon,\varepsilon_0>0$ such that if
the smallness conditions \eqref{6.2} and
\begin{equation}\label{6:small.1}\begin{aligned}
&<\bff_S>^w_{5/2,1/2+2\delta}
+ <\bff_\perp>_{p, 2+\delta} +<\bF_\perp>_{p,1+\delta}
+\|\bh\|_{\rT_{p,q}(\Gamma\times\BT)}
 < \varepsilon^2|\kappa|^{2\delta}
\end{aligned}\end{equation}
are satisfied,
then problem \eqref{eq:systemref} with $\kappa \not= 0$ admits a unique solution $(\bw,\fp)$ 
with
$$
\bw \in \rH^1_p(\BT, \rL_q(\Omega)^3) \cap \rL_p(\BT, \rH^2_q(\Omega)^3), \quad
\fp \in \rL_p(\BT, \hat\rH^1_q(\Omega))
$$
satisfying the estimate
\begin{align*}
<\bw_S>^w_{1,\delta} 
&+ <\nabla \bw_S>^w_{3/2, 1/2+\delta}
+<\bw_\perp>_{p, 1+\delta} + <\nabla\bw_\perp>_{p, 2+\delta} 
\\
&\quad
+ \|\bw\|_{\rL_p(\BT,\rH^2_q(\Omega))}
+ \|\pd_t\bw\|_{\rL_p(\BT, \rL_q(\Omega))} 
+ \|\nabla\fp\|_{\rL_p(\BT, \rL_q(\Omega))}
\leq \varepsilon|\kappa|^{2\delta}.
\end{align*}
\end{thm}

Alternatively, 
the case $\kappa\neq0$ allows to avoid 
spaces of functions with suitable pointwise decay,
such that the spatial asymptotics are merely quantified in terms of integrability.
However, the steady-state part of the velocity field
only belongs to suitable homogeneous Sobolev spaces.

\begin{thm}\label{mainthm:Oseen.int} 
Let $\kappa_0>0$
and $\delta\in(0,1)$, 
and assume~\eqref{eq:assumptions}.  
Let $\bff= \bff_S + \bff_\perp$ and $1<s<4/3$.
Then there exist
$\varepsilon,\varepsilon_0>0$ such that 
if the smallness conditions \eqref{6.2} and
\begin{equation}\label{6:small.Os.int}\begin{aligned}
&\|\bff\|_{\rL_p(\BT, \rL_s(\Omega))}
+\|\bff\|_{\rL_p(\BT, \rL_q(\Omega))}
+\|\bh\|_{\rT_{p,q}(\Gamma\times\BT)}
\leq \varepsilon^2 |\kappa|^{1/(1+\delta)}
\end{aligned}\end{equation}
are satisfied,
then problem \eqref{eq:systemref} with $\kappa \not= 0$ admits a unique solution $(\bw,\fp)$ 
with $\bw=\bw_S+\bw_\perp$ and
$$
\bw_S\in\hat\rH^2(\Omega)^3,\quad
\bw_\perp \in \rH^1_p(\BT, \rL_q(\Omega)^3) \cap \rL_p(\BT, \rH^2_q(\Omega)^3), \quad
\fp \in \rL_p(\BT, \hat\rH^1_q(\Omega))
$$
satisfying the estimate
\[
\begin{aligned}
&\|\nabla^2\bw_S\|_{\rL_s(\Omega)}
+ |\kappa|^{1/4}\|\nabla\bw_S\|_{\rL_{4s/(4-s)}(\Omega)}
+ |\kappa|^{1/2}\|\bw_S\|_{\rL_{2s/(2-s)}(\Omega)}
+ |\kappa|\,\|\partial_1\bw_S\|_{\rL_s(\Omega)}
\\
&\qquad
+\|\nabla^2\bw_S\|_{\rL_q(\Omega)}
+ \|\pd_t\bw_\perp\|_{\rL_p(\BT, \rL_s(\Omega))}
+ \|\bw_\perp\|_{\rL_p(\BT, \rH^2_s(\Omega))}
+ \|\nabla\fp\|_{\rL_p(\BT, \rL_s(\Omega))}
\\
&\qquad
+ \|\pd_t\bw_\perp\|_{\rL_p(\BT, \rL_q(\Omega))}
+ \|\bw_\perp\|_{\rL_p(\BT, \rH^2_q(\Omega))}
+ \|\nabla\fp\|_{\rL_p(\BT, \rL_q(\Omega))}
\leq \varepsilon \,|\kappa|^{1/2}.
\end{aligned}
\]
\end{thm}

Theorem~\ref{mainthm:Stokes}, Theorem~\ref{mainthm:Oseen}
and Theorem~\ref{mainthm:Oseen.int} will be proved in Section~\ref{sec:nonlin}.

\begin{remark}
\label{rem:noslip}
While Theorem~\ref{mainthm:Stokes}
deals with the case $\kappa=0$ of vanishing translational velocity,
Theorem~\ref{mainthm:Oseen} and Theorem~\ref{mainthm:Oseen.int}
yield existence of time-periodic solutions for arbitrary large $\kappa\neq 0$
if the data are sufficiently small.
However,
the treatment of no-slip boundary conditions
requires to take $|\kappa|$ small.
Indeed, as explained in Remark~\ref{rem:noslip.formula},
no-slip conditions
are expressed by boundary data $\bh$ of the form~\eqref{eq:noslipdata}.
In virtue of \eqref{6.2}, \eqref{5.4} and \eqref{5.4*},
we can then estimate
\[
\|\bh\|_{\rT_{p,q}(\Gamma\times\BT)}
\leq C(1+\varepsilon_0) (\varepsilon_0+|\kappa|).
\]
Therefore, the smallness conditions~\eqref{6:small.1}
and~\eqref{6:small.Os.int}
can be satisfied by fixing
$\kappa_0>0$ and choosing $|\kappa|$ and $\varepsilon_0>0$ sufficiently small.
\end{remark}

\subsection{The associated linear problems}
\label{subsec:linresults}

In the case $\kappa=0$, system~\eqref{eq:nslin.tp}
reduces to the time-periodic Stokes equations in an 
exterior domain. 
The following existence theorem
was shown in~\cite{EiterKyedShibata_PeriodicLp}.
We set $\rL_{q,3b}(\Omega)=\{f\in\rL_{q}(\Omega)\mid\mathrm{supp}\, f\subset B_{3b}\}$
to shorten the notation.

\begin{thm}\label{thm:tpStokes}
Let $\kappa=0$.
Let $1 < p < \infty$, $3<q<\infty$ and $\ell\in(0,3]$.
For all $\bff=\bff_S+\bff_\perp$
such that $\bff_S=\dv \bF_S+\bg_S$ and $\bff_\perp=\dv\bF_\perp+\bg_\perp$
with $\bg=\bg_S+\bg_\perp\in\rL_p(\BT, \rL_{q,3b}(\Omega)^3)$ and
\[
<\bF_S>_2 + <\dv\bF_S>_3 + <\bF_\perp>_{p, \ell} + <\dv\bF_\perp>_{p, \ell+1} < \infty,
\]
and for all $\bh\in\rT_{p,q}(\Gamma\times\BT)$,
problem \eqref{eq:nslin.tp} with $\kappa=0$
admits a unique solution $(\bv,\fp)$ with
$$\bv \in \rH^1_p(\BT, \rL_q(\Omega)^3) \cap 
\rL_p(\BT, \rH^2_q(\Omega)^3), \quad
\fp \in \rL_p(\BT, \hat\rH^1_q(\Omega)),
$$
possessing the estimate
\begin{equation}
\begin{aligned}
&\|\pd_t\bv\|_{\rL_p(\BT, \rL_q(\Omega))}
+ \|\bv\|_{\rL_p(\BT, \rH^2_q(\Omega))}
+ \|\nabla\fp\|_{\rL_p(\BT, \rL_q(\Omega))}
\\
&\quad
+<\bv_S>_1 + <\nabla\bv_S>_2 
+<\bv_\perp>_{p, \ell}  
+<\nabla\bv_\perp>_{p, \ell+1}
\\
&\qquad\quad \leq C\big(
<\dv\bF_S>_3 + <\bF_S>_2
+<\dv\bF_\perp>_{p, \ell+1} 
+ <\bF_\perp>_{p, \ell} 
\\
&\qquad\qquad\qquad\qquad\qquad\qquad\qquad\qquad
+ \|\bg\|_{\rL_p(\BT, \rL_q(\Omega))}
+\|\bh\|_{\rT_{p,q}(\Gamma\times\BT)}\big).
\end{aligned}
\label{est:tpStokes}
\end{equation}
Here the constant $C>0$ only depends on $\Omega$, $\CT$, $\mu$, $p$, $q$, and $\ell$.
\end{thm} 

If $\kappa\neq 0$,
then \eqref{eq:nslin.tp} is a time-periodic Oseen problem,
and we have to take into account the anisotropic spatial behavior of solutions.
In this case we shall derive the following result 
on existence of solutions with suitable pointwise decay.

\begin{thm}\label{thm:tpOseen}
Let $0<|\kappa|\leq\kappa_0$.
Let $1 < p < \infty$, $3<q<\infty$, $\delta\in(0,\frac{1}{4})$ and $\ell\in(0,3]$.
For all $\bff=\bff_S+\bff_\perp$
such that $\bff_S=\tilde\bff_S+\bg_S$ and $\bff_\perp=\dv\bF_\perp+\bg_\perp$
with $\bg=\bg_S+\bg_\perp\in\rL_p(\BT, \rL_{q,3b}(\Omega)^3)$ and
\[
<\tilde\bff_S>^w_{5/2, 1/2+2\delta}
+<\dv\bF_\perp>_{p, 1+\ell} 
+ <\bF_\perp>_{p, \ell} < \infty,
\]
and for all
$\bh\in\rT_{p,q}(\Gamma\times\BT)$,
problem \eqref{eq:nslin.tp}
admits a unique solution $(\bv,\fp)$ with
$$\bv \in \rH^1_p(\BT, \rL_q(\Omega)^3) \cap 
\rL_p(\BT, \rH^2_q(\Omega)^3), \quad
\fp \in \rL_p(\BT, \hat\rH^1_q(\Omega)),
$$
possessing the estimate
\begin{equation}
\begin{aligned}
&\|\pd_t\bv\|_{\rL_p(\BT, \rL_q(\Omega))}
+ \|\bv\|_{\rL_p(\BT, \rH^2_q(\Omega))}
+ \|\nabla\fp\|_{\rL_p(\BT, \rL_q(\Omega))}
\\
&\quad
+|\kappa|^\delta<\bv_S>_{1,\delta}^w 
+\ |\kappa|^\delta<\nabla\bv_S>_{3/2,1/2+\delta}^w 
+<\bv_\perp>_{p, \ell}  
+<\nabla\bv_\perp>_{p, \ell+1}
\\
&\qquad\quad
\leq C\big(
<\tilde\bff_S>^w_{5/2, 1/2+2\delta}
+<\dv\bF_\perp>_{p, \ell+1} 
+ <\bF_\perp>_{p, \ell} 
\\
&\qquad\qquad\qquad\qquad\qquad\qquad\qquad
+ \|\bg\|_{\rL_p(\BT, \rL_q(\Omega))}
+\|\bh\|_{\rT_{p,q}(\Gamma\times\BT)}\big).
\end{aligned}
\label{est:tpOseen}
\end{equation}
Here the constant $C>0$ only depends on $\Omega$, $\CT$, $\mu$, $p$, $q$, $\delta$, $\ell$ and $\kappa_0$.
\end{thm} 

Alternatively,
the following well-posedness result for $\kappa\neq 0$
does not quantify the decay of 
the data and the solutions
in a pointwise sense,
but merely uses (homogeneous) Sobolev spaces.

\begin{thm}\label{thm:tpOseen.int}
Let $1 < p < \infty$, $1<s<2$.
For all $\bff\in\rL_p(\BT,\rL_s(\Omega)^3)$ and
$\bh\in\rT_{p,s}(\Gamma\times\BT)$
problem \eqref{eq:nslin.tp} with $\kappa\neq 0$
admits a unique solution $(\bv,\fp)$ with 
$\bv=\bv_S+\bv_\perp$ satisfying
$$\bv_S\in\hat\rH^2_s(\Omega)^3\cap\rL_{2s/(2-s)}(\Omega)^3,
\quad \bv_\perp \in \rH^1_p(\BT, \rL_s(\Omega)^3) \cap 
\rL_p(\BT, \rH^2_s(\Omega)^3), \quad
\fp \in \rL_p(\BT, \hat\rH^1_s(\Omega)),
$$
possessing the estimate
\begin{equation}
\begin{aligned}
\|\nabla^2\bv_S\|_{\rL_s(\Omega)}
&+|\kappa|^{1/4}\|\nabla\bv_S\|_{\rL_{4s/(4-s)}(\Omega)}
+|\kappa|^{1/2}\|\bv_S\|_{\rL_{2s/(2-s)}(\Omega)}
+|\kappa|\,\|\partial_1\bv_S\|_{\rL_s(\Omega)}
\\
&\qquad
+\|\pd_t\bv_\perp\|_{\rL_p(\BT, \rL_s(\Omega))}
+ \|\bv_\perp\|_{\rL_p(\BT, \rH^2_s(\Omega))}
+ \|\nabla\fp\|_{\rL_p(\BT, \rL_s(\Omega))}
\\
&\qquad\qquad\qquad\qquad\qquad\qquad
\leq C\big(
\|\bff\|_{\rL_p(\BT, \rL_s(\Omega))}
+\|\bh\|_{\rT_{p,s}(\Gamma\times\BT)}\big).
\end{aligned}
\label{est:tpOseen.int}
\end{equation}
If additionally $\bff\in\rL_p(\BT,\rL_q(\Omega)^3)$ and
$\bh\in\rT_{p,q}(\Gamma\times\BT)$ 
for some $q\in(1,\infty)$, 
then 
\begin{equation}
\begin{aligned}
&\|\nabla^2\bv_S\|_{\rL_q(\Omega)}
+|\kappa|\,\|\partial_1\bv_S\|_{\rL_q(\Omega)}
\leq C\big(
\|\bff\|_{\rL_p(\BT, \rL_q(\Omega))}
+\|\bh\|_{\rT_{p,q}(\Gamma\times\BT)}\big).
\end{aligned}
\label{est:tpOseen.int2}
\end{equation}
If $|\kappa|\leq\kappa_0$ and $1<s<3/2$, 
then the constant $C>0$ only depends on $\Omega$, $\CT$, $\mu$, $p$, $q$, $s$ and $\kappa_0$.
\end{thm} 

For a proof of Theorem~\ref{thm:tpStokes}
we refer to \cite[Theorem 5.2]{EiterKyedShibata_PeriodicLp}.
The derivation of Theorem~\ref{thm:tpOseen} and Theorem~\ref{thm:tpOseen.int}
is the scope of Section~\ref{sec:Oseen}.

\section{The time-periodic Oseen problem}
\label{sec:Oseen}

In this section we 
show existence of solutions to the time-periodic Oseen problem~\eqref{eq:nslin.tp} for $\kappa\neq0$
with suitable decay properties
as stated in Theorem~\ref{thm:tpOseen}. 
To this end, 
we decompose all functions into a steady-state part
and a purely oscillatory part
according to~\eqref{eq:decomposition}. 
Due to the linearity of the system, 
this leads to two problems,
which we examine in the case of homogeneous boundary conditions.
Firstly, we obtain the steady-state problem
\begin{equation}\label{eq:3.1}
-\mu \Delta \bu - \kappa\pd_1\bu + \nabla\fp = \bff_S, \quad
\dv \bu = 0 \quad\text{in $\Omega$}, \quad
\bu|_\Gamma = 0,
\end{equation}
that is, we study time-independent solutions to~\eqref{eq:nslin.tp}.
Secondly, we consider purely oscillatory solutions to~\eqref{eq:nslin.tp},
which
leads to the problem
\begin{equation}\label{eq:osc}
\pd_t\bv_\perp - \mu\Delta\bv_\perp -\kappa \pd_1\bv_\perp+ \nabla\fp_\perp 
= \bff_\perp,
\quad \dv\bv_\perp = 0\quad 
\text{in $\Omega\times\BT$}, \quad 
\bv_\perp|_{\Gamma\times\BT}  = 0.
\end{equation}
Here the subscript $\bot$ means that all functions have vanishing time mean.
We study these problems separately. 
A combination of the results will lead to a proofs of Theorem~\ref{thm:tpOseen}
and Theorem~\ref{thm:tpOseen.int}.
As explained in Remark~\ref{rem:noslip},
it will be important to derive estimates 
where the constants are independent of $\kappa$
for $|\kappa|\leq\kappa_0$.

\subsection{Existence of time-independent solutions}

For the steady-state Oseen problem~\eqref{eq:3.1},
existence of solutions is guaranteed by the following result. 
It characterizes the solution in terms of integrability properties 
and pointwise decay
for suitable forcing terms.

\begin{thm}\label{thm:stat} 
Let $3 < q < \infty$, 
$0 < \delta < 1/4$ and $0<|\kappa|\leq \kappa_0$. 
Let $\bff_S = \tilde\bff + \bg$, where  
$<\tilde\bff>^w_{5/2, 1/2+2\delta} < \infty$ and $\bg \in L_{q, 3b}(\Omega)^3$. 
Then, problem \eqref{eq:3.1} admits a unique solution $(\bu,\fp)\in \rH^2_q(\Omega)^3\times \rH^1_q(\Omega)$ possessing the estimate
\begin{equation}\label{est:statOseen}
\|\bu\|_{\rH^2_q(\Omega)}
+|\kappa|^\delta<\bu>^w_{1,\delta} + |\kappa|^\delta<\nabla \bu>^w_{3/2, 1/2+\delta} 
+ \|\fp\|_{\rH^1_q(\Omega)}
\leq C
\big(<\tilde\bff>^w_{5/2, 1/2+2\delta} + \|\bg\|_{\rL_q(\Omega)}\big)
\end{equation}
for some constant $C>0$ depending solely  on $\Omega$, $\mu$, $q$, $\delta$ and $\kappa_0$.
\end{thm}

\begin{proof}
We first show that $\bff_S\in\rL^1(\Omega)^3\cap\rL^q(\Omega)^3$
and
\begin{equation}\label{eq:3.2}\|\tilde\bff\|_{\rL_1(\Omega)} + 
\|\tilde\bff\|_{\rL_q(\Omega)} \leq C_\delta<\tilde\bff>^w_{5/2, 1/2+2\delta}.
\end{equation}
In fact, using the polar coordinates $x_1=r\cos\theta$, $x_2= r\sin\theta
\cos\varphi$, $x_3 = r\sin\theta\sin\varphi$ for $r>0$, $0 \leq \theta < \pi$ and 
$0 \leq \varphi < 2\pi$
and a change of variables, we obtain
\begin{align*}
\|\tilde\bff\|_{\rL_1(\Omega)} &\leq \ <\tilde\bff>^w_{5/2, 1/2+2\delta}
\int_{\BR^3}(1+|x|)^{-5/2}(1+|x|-x_1)^{-1/2-2\delta}\,\dd x \\
& \leq 4\pi\,<\tilde\bff>^w_{5/2, 1/2+2\delta}\int^\infty_0\int^{\pi/2}_0(1+r)^{-5/2}
\frac{r^2\sin\theta}{(1 + r(1-\cos\theta))^{1/2+2\delta}}\,\dd r\,\dd\theta
\\
& = 4\pi\,<\tilde\bff>^w_{5/2, 1/2+2\delta}\int^\infty_0\int^1_0\frac{(1+r)^{-5/2}r^2}
{(1+rt)^{1/2+2\delta}}\,\dd r \dd t \\
& \leq C_\delta<\tilde\bff>^w_{5/2, 1/2+2\delta}
\int^\infty_0(1+r)^{-5/2}(1+r)^{1/2-2\delta}
r\,\dd r \\
&\leq C_\delta<\tilde\bff>^w_{5/2, 1/2+2\delta}.
\end{align*}
Since we clearly have $\|\tilde\bff\|_{\rL_\infty(\Omega)}\leq \ <\tilde\bff>_{5/2,1/2+2\delta}^w$, 
we thus conclude 
\eqref{eq:3.2}.
Now the estimates of $\|\bu\|_{\rH^2_q(\Omega)}$ and $\|\fp\|_{\rH^1_q(\Omega)}$
are a direct consequence of~\cite[Theorem 3.1]{YS99},
and the remaining estimates follow from~\cite[Theorem 4.1]{YS99}.
\end{proof}

In contrast to the previous framework, 
where decay is specified in a pointwise sense,
one can also show existence within 
homogeneous Sobolev spaces.

\begin{thm}\label{thm:stat.Lq} 
Let $1 < s < 3/2$, $1<q<\infty$ and $0<|\kappa|\leq \kappa_0$. 
Let $\bff_S \in\rL_q(\Omega)^3\cap\rL_s(\Omega)^3$. 
Then, problem \eqref{eq:3.1} admits a unique solution $(\bu,\fp)\in \hat\rH^2_{q}(\Omega)^3\times \hat\rH^1_{q}(\Omega)$ possessing the estimate
\begin{equation}\label{est:statOseen.Lq}
\begin{aligned}
&\|\nabla^2\bu\|_{\rL_s(\Omega)}
+\|\nabla^2\bu\|_{\rL_q(\Omega)}
+|\kappa|^{1/4}\|\nabla\bu\|_{\rL_{4s/(4-s)}(\Omega)}
+|\kappa|^{1/2}\|\bu\|_{\rL_{2s/(2-s)}(\Omega)}
\\
&\
+|\kappa|\,\|\partial_1\bu\|_{\rL_s(\Omega)}
+\|\nabla\fp\|_{\rL_s(\Omega)}
+\|\nabla\fp\|_{\rL_q(\Omega)}
+\|\fp\|_{\rL_{3s/(3-s)}(\Omega)}
\leq C\big(\|\bff_S\|_{\rL_s(\Omega)}+\|\bff_S\|_{\rL_q(\Omega)}\big)
\end{aligned}
\end{equation}
for some constant $C>0$ depending solely  on $\Omega$, $\mu$, $q$, $s$ and $\kappa_0$.
\end{thm}

\begin{proof}
Let us assume $s\leq q$. Otherwise, we reverse the role of $s$ and $q$. 
Then~\cite[Theorem 2.1]{Galdi_Oseen1992} implies the existence of 
a unique solution $(\bu,\fp)$ 
satisfying 
\begin{equation}
\begin{aligned}
\|\nabla^2\bu\|_{\rL_s(\Omega)}
&+|\kappa|^{1/4}\|\nabla\bu\|_{\rL_{4s/(4-s)}(\Omega)}
+|\kappa|^{1/2}\|\bu\|_{\rL_{2s/(2-s)}(\Omega)}
\\
&\quad
+|\kappa|\,\|\partial_1\bu\|_{\rL_s(\Omega)}
+\|\nabla\fp\|_{\rL_s(\Omega)}
+\|\fp\|_{\rL_{3s/(3-s)}(\Omega)}
\leq C\|\bff_S\|_{\rL_s(\Omega)}
\end{aligned}
\label{est:statOseen.Lq.1}
\end{equation}
with $C>0$ independent of $\kappa$.
We now consider $(\bu,\fp)$ as a solution to the Stokes problem
\[
-\mu \Delta \bu + \nabla\fp = \bff_S + \kappa\pd_1\bu , \quad
\dv \bu = 0 \quad\text{in $\Omega$}, \quad
\bu|_\Gamma = 0,
\]
with right-hand side $\bff + \kappa\pd_1\bu\in\rL_r(\Omega)$
with $r=\min\{q,\frac{4s}{4-s}\}$
since $\partial_1\bu\in\rL_s(\Omega)\cap\rL_{4s/(4-s)}(\Omega)$ 
by the previous estimate.
From~\cite[Theorem V.4.8]{Galdi} we conclude
\[
\begin{aligned}
\|\bu\|_{\rL_r(\Omega)}+\|\nabla\fp\|_{\rL_r(\Omega)}
&\leq C\big( \|\bff_S\|_{\rL_{r}(\Omega)} + |\kappa|\, \|\partial_1\bu\|_{\rL_{r}(\Omega)}\big)
\\
&\leq C\big( \|\bff_S\|_{\rL_{r}(\Omega)}
+ |\kappa|\, \|\partial_1\bu\|_{\rL_{s}(\Omega)}
+|\kappa|\, \|\partial_1\bu\|_{\rL_{4s/(4-s)}(\Omega)}\big)
\\
&\leq  C\big( \|\bff_S\|_{\rL_{r}(\Omega)} + \|\bff_S\|_{\rL_{s}(\Omega)}),
\end{aligned}
\]
where we used estimate~\eqref{est:statOseen.Lq.1} as well as $r\geq s$ and $|\kappa|\leq \kappa_0$.
If $r=q$, this completes the proof.
If $r<q$, then $r=4s/(4-s)$, and 
Sobolev's inequality and classical interpolation 
implies $\nabla\bu\in\rL_{\tilde r}(\Omega)$
for $\tilde r\in[r,3r/(3-r)]$ if $r<3$
and for $\tilde r\in[r,\infty)$ if $r>3$. 
By repeating the above argument with $r$ replaced with $\tilde r$,
an iteration will finally lead to the asserted estimate for $r=q$.
\end{proof}

\subsection{Existence of purely oscillatory solutions}\label{subsec:2}

Existence of solutions to the purely oscillatory problem~\eqref{eq:osc}
is guaranteed by the following theorem. 
Observe that here it is not necessary to distinguish between the cases 
$\kappa=0$ and $\kappa\neq0$.
\begin{thm} \label{thm:osc} 
Let $1 < p, q < \infty$ and $\kappa \in \BR$ with
$|\kappa|\leq\kappa_0$. Then, 
for any $\bff_\perp \in \rL_p(\BT, \rL_q(\Omega)^3)$
with $\int_\BT \bff_\perp(\cdot, s)\,\dd s=0$, problem \eqref{eq:osc}
admits a solution $(\bv_\perp,\fp_\perp)$ with
$$\bv_\perp \in \rH^1_p(\BT, \rL_q(\Omega)^3) \cap 
\rL_p(\BT, \rH^2_q(\Omega)^3), \enskip 
\fp_\perp \in \rL_p(\BT, \hat\rH^1_q(\Omega)),
\enskip \int_\BT\bv_\perp(\cdot, s)\,\dd s= 0, 
\enskip \int_\BT\fp_\perp(\cdot, s)\,\dd s=0,
$$
which satisfies the estimate
\begin{equation}\label{est:osc}
\|\pd_t\bv_\perp\|_{\rL_p(\BT, \rL_q(\Omega))}
+ \|\bv_\perp\|_{\rL_p(\BT, \rH^2_q(\Omega))}
+ \|\nabla\fp_\perp\|_{\rL_p(\BT, \rL_q(\Omega))}
\leq C\|\bff_\perp\|_{\rL_p(\BT, \rL_q(\Omega))}
\end{equation}
for some constant $C>0$ only depending on $\Omega$, $\CT$, $\mu$, $p$, $q$ and $\kappa_0$.

If $(\tilde\bv_\perp,\tilde\fp_\perp)$ is another solution to \eqref{eq:osc} with
$$
\tilde\bv_\perp \in \rH^1_r(\BT, \rL_s(\Omega)^3) \cap 
\rL_r(\BT, \rH^2_s(\Omega)^3), \enskip 
\tilde\fp_\perp \in \rL_r(\BT, \hat\rH^1_s(\Omega)),
$$
for some $1<r,s<\infty$,
then $\bv_\perp=\tilde\bv_\perp$ and $\nabla\fp_\perp=\nabla\tilde\fp_\perp$.
\end{thm}

To prove Theorem~\ref{thm:osc}, we 
use the method recently introduced in~\cite{EiterKyedShibata_PeriodicLp},
which is based on the existence of suitable $\sR$-bounds for solution operators to 
the corresponding resolvent problem
\begin{equation}\label{eq:2.1}
\lambda \bw - \mu \Delta \bw - \kappa \pd_1\bw + \nabla\fr = \bff,
\quad \dv \bw=0
\quad\text{in $\Omega$}, \quad 
\bw|_\Gamma  = 0.
\end{equation}
For this problem, we have the following theorem.

\begin{thm}\label{thm:res} 
Let $1 < q < \infty$
and $0\leq|\kappa|\leq\kappa_0$.  
Let
\begin{equation}
\label{eq:res.set}
\rho[\kappa]:=
\begin{cases}
\BC\setminus(-\infty,0] &\text{if } \kappa=0,
\\
\big\{\lambda\in\BC\mid |\kappa|\,\mathrm{Re}(\lambda)+\mathrm{Im}(\lambda)^2>0\big\} &\text{if } \kappa\neq0,
\end{cases}
\end{equation}
There exist operator families 
$(\sS_\kappa(\lambda)) \subset\sL(\rL_q(\Omega)^3,\rH^2_q(\Omega)^3)$
and $(\sP_\kappa(\lambda)) \subset \sL(\rL_q(\Omega)^3, \hat \rH^1_q(\Omega)))$
such that for every 
$\lambda\in \rho[\kappa]$
and every $\bff \in \rL_q(\Omega)^3$
the pair $(\bw,\fr)=(\sS_\kappa(\lambda)\bff,\sP_\kappa(\lambda)\bff)$
is the unique solution to \eqref{eq:2.1}
and satisfies the estimate
\begin{equation}\label{est:nslin.res}
|\lambda|\,\| \sS_\kappa(\lambda)\bff\|_{\rL_q(\Omega)}
+\|\nabla^2 \sS_\kappa(\lambda)\bff\|_{\rL_q(\Omega)}
+ \|\nabla\sP_\kappa(\lambda)\bff\|_{\rL_q(\Omega)}
\leq C\|\bff\|_{\rL_q(\Omega)}.
\end{equation}
If $0<\varepsilon<\pi/2$ and 
$\delta>0$ such that $\lambda\in\Sigma_{\varepsilon,\delta}$, 
then $C$ only depends on $\Omega$, $q$, $\varepsilon$, $\delta$ and $\kappa_0$.
Moreover, there exist constants $\lambda_0, r_{0}>0$, 
depending on $\Omega$, $\mu$, $q$, $\varepsilon$ and $\kappa_0$,
such that
$$\sS_\kappa \in {\rm Hol}\,(\Sigma_{\varepsilon, \lambda_0}, \sL(\rL_q(\Omega)^3, 
\rH^2_q(\Omega)^3)), \quad 
\sP_\kappa \in {\rm Hol}\,(\Sigma_{\varepsilon, \lambda_0}, \sL(\rL_q(\Omega)^3, 
\hat \rH^1_q(\Omega))),
$$
and
\[
\begin{aligned}
\sR_{\sL(\rL_q(\Omega)^3, \rH^{2-j}_q(\Omega)^3)}
(\{(\lambda\pd_\lambda)^\ell(\lambda^{j/2}\sS_\kappa(\lambda)) \mid \lambda \in \Sigma_{\varepsilon, \lambda_0}\})
&\leq r_{0}, \\
\sR_{\sL(\rL_q(\Omega)^3, \rL_q(\Omega)^3)}
(\{(\lambda\pd_\lambda)^\ell(\nabla\sP_\kappa(\lambda)) \mid \lambda \in \Sigma_{\varepsilon, \lambda_0}\})
&\leq r_{0}
\end{aligned}
\]
for $\ell=0,1$, $j=0,1,2$.
\end{thm}

\begin{proof}
For $\kappa=0$, the result follows mainly from 
\cite[Theorem 1.6]{Shibata14} 
and \cite[Theorem 9.1.4]{Shibata16}
as was shown in \cite[Theorem 4.2]{EiterKyedShibata_PeriodicLp}.
If $\kappa\neq0$,
the existence of the solution operators
$\sS_\kappa(\lambda)$ 
and $\sP_\kappa(\lambda)$ 
together with the estimate~\eqref{est:nslin.res}
for all $\lambda\in\Sigma_{\varepsilon}\setminus\{0\}$
was derived in~\cite[Theorem 4.4]{KS}.
Since the term $\kappa\pd_1$ can be regarded as a perturbation of
the Stokes operator, which is uniform for $|\kappa|\leq \kappa_0$,
the asserted results on the analyticity and the $\sR$-bounds follow
from those for $\kappa=0$
if $\lambda_0$ is taken sufficiently large.
\end{proof}

To show existence of solutions to~\eqref{eq:osc}, 
we combine 
the $\sR$-bounds from Theorem~\ref{thm:res}
with the following multiplier theorem.

\begin{thm}\label{thm:multiplier}
Let $X$ and $Y$ be UMD spaces, and let
$
M \in \rL_\infty(\BR, \sL(X, Y))
\cap C^{1}(\BR, \sL(X, Y))
$
satisfy 
\begin{equation}\label{est:Rbounds.multiplier}
\sR_{\sL(X, Y)}\bigl\{ M(t) \mid t\in\BR\setminus\{0\}\bigr\} \leq r_0,
\qquad
\sR_{\sL(X, Y)}\bigl\{t M'(t) \mid t\in\BR\setminus\{0\}\bigr\}\leq r_0,
\end{equation}
for some $r_0>0$.
Then $M|_\BZ$ is an $\rL_p(\BT)$-multiplier such that
\begin{equation}\label{est:multipliernorm}
\forall f\in C^\infty(\BT;X):
\qquad
\|\sF^{-1}_\BT[M|_\BZ\,\sF_\BT[f]]\|_{\rL_p(\BT; Y)} 
\leq C_p r_0
\|f\|_{\rL_p(\BT; X)}
\end{equation} 
for some constant $C_p>0$ only depending on $p$.
\end{thm}

\begin{proof}
The result was derived in \cite[Corollary 2.3]{EiterKyedShibata_PeriodicLp}
as a combination
of an operator-valued transference principle for multipliers 
(see \cite[Prop.5.7.1]{HytonenVNeervenVeraarWeis2016})
with the multiplier theorem due to Weis \cite[Theorem 3.4]{Weis2001}.
\end{proof}

For a proof of Theorem~\ref{thm:res} 
we now follow the approach from~\cite{EiterKyedShibata_PeriodicLp}.

\begin{proof}[Proof of Theorem~\ref{thm:res}]
Let $\varphi\in C^\infty(\BR)$ with 
$\varphi(\sigma)=1$ for $\sigma\geq \lambda_0+1/2$ and 
$\varphi(\sigma)=0$ for $\sigma \leq \lambda_0+1/4$.
Set 
$\bff_h = \sF^{-1}_\BT[\varphi(|\tfrac{2\pi}{\CT}k|)\sF_\BT[\bff_\perp](k)]$
and 
$$\bv_h = \sF^{-1}_\BT\big[\sS_\kappa(i\tfrac{2\pi}{\CT}k)\varphi(|\tfrac{2\pi}{\CT}k|)\sF_\BT[\bff_\perp](k)\big],
\qquad \fp_h = \sF^{-1}_\BT\big[\sP_\kappa(i\tfrac{2\pi}{\CT}k)\varphi(|\tfrac{2\pi}{\CT}k|)\sF_\BT[\bff_\perp](k)\big],
$$
where $\lambda_0$, $\sS_\kappa$ and $\sP_\kappa$ are given in Theorem~\ref{thm:res}.
Then $\bv_h$ and $\fp_h$ satisfy the equations
$$\pd_t\bv_h - \mu\Delta\bv_h - \kappa\partial_1\bv_h + \nabla\fp_h 
= \bff_h,
\quad \dv\bv_h = 0\quad
\text{in $\Omega\times\BT$}, \quad 
\bv_h|_{\Gamma\times\BT}  = 0.
$$
Moreover, from the $\sR$-bounds from Theorem~\ref{thm:res} we derive
\[
\begin{aligned}
\sR_{\sL(\rL_q(\Omega)^3, \rH^{2-j}_q(\Omega)^3)}
\big(\big\{(\lambda\pd_\lambda)^\ell(\lambda^{j/2}\varphi(|\lambda|)\sS_\kappa(\lambda)) \mid \lambda \in \Sigma_{\varepsilon, \lambda_0}\big\}\big)
&\leq C_\varphi r_{0}, \\
\sR_{\sL(\rL_q(\Omega)^3, \rL_q(\Omega)^3)}
\big(\big\{(\lambda\pd_\lambda)^\ell(\varphi(|\lambda|)\nabla\sP_\kappa(\lambda)) \mid \lambda \in \Sigma_{\varepsilon, \lambda_0}\big\}\big)
&\leq C_\varphi r_{0}
\end{aligned}
\]
for $\ell=0,1$, $j=0,1,2$,
where $C_\varphi$ is a constant only depending on $\varphi$.
Using Theorem~\ref{thm:multiplier},
we conclude
\begin{equation}\label{est:highfr}
\|\pd_t\bv_h\|_{\rL_p(\BT, \rL_q(\Omega))} 
+ \|\bv_h\|_{\rL_p(\BT, \rH^2_q(\Omega))}
+ \|\nabla\fp_h\|_{\rL_p(\BT, \rL_q(\Omega))}
\leq C\|\bff_h\|_{\rL_p(\BT, \rL_q(\Omega))}
\leq C\|\bff_\perp\|_{\rL_p(\BT, \rL_q(\Omega))}.
\end{equation}
We now set 
\[
\begin{aligned}
\bv_\bot(t)&= \bv_h(t) + \sum_{0<|k| \leq \lambda_0}\e^{i\frac{2\pi}{\CT}kt}
\sS_\kappa(i\tfrac{2\pi}{\CT}k)\sF_\BT[\bff_\perp](k), \\
\fp_\bot(t)&= \fp_h(t) + \sum_{0<|k| \leq \lambda_0}\e^{i\frac{2\pi}{\CT}kt}
\sP_\kappa(i\tfrac{2\pi}{\CT}k)\sF_\BT[\bff_\perp](k).
\end{aligned}
\]
Then,  $\bv_\bot$ and $\fp_\bot$ satisfy 
\eqref{eq:osc}, 
and from \eqref{est:nslin.res} and \eqref{est:highfr}
we conclude estimate \eqref{est:osc}.

For the uniqueness statement,
consider the difference 
$(\bu_\bot,\fq_\bot)=(\bv_\bot-\tilde\bv_\bot,\fp_\bot-\tilde\fp_\bot)$,
which is a solution to~\eqref{eq:osc} with $\bff_\perp=0$.
Then, for each $k \in \BZ\setminus\{0\}$, the functions 
$\hat\bu_k = \sF_\BT[\bu_\perp](k)$ and $\hat\fq_k = \sF_\BT[\fq_\perp](k)$
satisfy $\hat\bu_k \in \rH^2_q(\Omega)^3+\rH^2_s(\Omega)^3$ and $\hat\fq_k\in \hat \rH^1_q(\Omega)+\hat \rH^1_s(\Omega)$ and solve the homogeneous equations
$$
i\tfrac{2\pi}{\CT}k\hat\bu_k - \mu\Delta\hat\bu_k
-\kappa\partial_1\hat\bu_k
+ \nabla\hat\fq_k
= 0,
\quad \dv\hat\bu_k = 0\quad
\text{in $\Omega$}, 
\qquad 
\hat\bu_k|_\Gamma  = 0.
$$
Using elliptic regularity for the Stokes operator
and Sobolev embeddings,
similarly as in the proof of Theorem~\ref{thm:stat.Lq}
we conclude
$\hat\bu_k=\nabla \hat\fq_k=0$
for any $k\in\BZ$.
This shows $\bv=\nabla \fp = 0$
and completes the proof.
\end{proof}

\subsection{Pointwise decay of the oscillatory part}\label{subsec:osc.decay}

In this section,
we study decay properties of solutions $(\bv_\perp,\fp_\perp)$
to the purely oscillatory problem~\eqref{eq:osc}.
We derive decay properties of  
$\|\bv_\perp(x, \cdot)\|_{\rL_p(\BT)}$
and $\|\nabla\bv_\perp(x, \cdot)\|_{\rL_p(\BT)}$
as $|x|\to\infty$
as stated in the following theorem. 

\begin{thm}\label{thm:osc.decay} 
Let $1<p<\infty$, $3<q<\infty$, $\ell\in(0,3]$ and $\kappa_0\geq 0$.
Let $\bff_\perp= \dv \bF_\perp+ \bg_\perp$ with
\begin{equation}\label{assump:4.1}\begin{aligned}
&\int_\BT\bF_\perp(x, t)\,\dd t=0,
&\quad &\!\!\!<\bF_\perp>_{p, \ell} + <\dv\bF_\perp>_{p, \ell+1}< \infty,  \\
&\int_\BT\bg_\perp(x, t)\,\dd t=0, &\quad 
&\bg_\perp \in \rL_p(\BT, \rL_{q, 3b}(\Omega)).
\end{aligned}\end{equation}
Let $(\bv_\bot,\fp_\bot)$ be the solution to 
\eqref{eq:osc} according to Theorem~\ref{thm:osc}.
Then,  $\bv_\perp$ satisfies
\begin{equation}\label{eq:4.2}\begin{aligned}
<\bv_\perp>_{p, \ell}  + 
<\nabla\bv_\perp>_{p, \ell+1} 
\leq C(<\dv\bF_\perp>_{p, \ell+1} 
+ <\bF_\perp>_{p, \ell}  + \|\bg_\perp\|_{\rL_p(\BT, \rL_q(\Omega))})
\end{aligned}\end{equation}
with some constant $C > 0$
only dependent on $\Omega$, $\CT$, $\mu$, $p$, $q$, $\ell$ and $\kappa_0$. 
\end{thm}

\begin{remark}\label{rem:osc.decay}
Since $3 < q < \infty$, we have
$\|\dv\bF_\perp\|_{\rL_p(\BT, \rL_q(\Omega))} \leq C_{q,\ell} <\dv\bF_\perp>_{p, \ell+1}$,
so that $\bff_\perp \in \rL_p(\BT, \rL_q(\Omega))$
and
\[
\|\bff_\perp\|_{\rL_p(\BT, \rL_q(\Omega))}
\leq C \big(
<\dv\bF_\perp>_{p, \ell+1} 
+ <\bF_\perp>_{p, \ell}  + \|\bg_\perp\|_{\rL_p(\BT, \rL_q(\Omega))}
\big).
\]
Therefore, existence of $(\bv_\perp,\fp_\perp)$ 
indeed follows from Theorem~\ref{thm:osc}.
\end{remark}

This pointwise estimate will be concluded 
by using the velocity fundamental solution to 
the purely oscillatory problem
\eqref{eq:osc},
which is a tensor field $\Gamma_\perp^\kappa$ 
such that $\bv_\perp:=\Gamma_\perp^\kappa\ast \bH_\perp$ defines a solution to \eqref{eq:osc} 
for $\Omega=\BR^3$.
We use the following properties of $\Gamma_\perp^\kappa$, which were mainly 
derived in~\cite{EK1} in the general multidimensional case.

\begin{thm}\label{lem:4.1} 
Let $\kappa\in\BR$ and $\mu,\CT>0$.
Let
\begin{equation}\label{kernel}
\Gamma_\perp^\kappa = \sF^{-1}_{\BR^3\times\BT}\Bigl[\frac{1-\delta_\BZ(k)}{\mu|\xi|^2
- i\kappa \xi_1 + i\tfrac{2\pi}{\CT}k}\Bigr({\rm I} - \frac{\xi\otimes\xi}{|\xi|^2}\Bigr)\Bigr].
\end{equation}
Then, it holds $\Gamma_\perp^\kappa \in \rL_q(\BR^3\times\BT)^{3\times 3}$ for $q \in (1, 5/3)$,
and $\pd_j\Gamma_\perp^\kappa  \in \rL_r(\BR^3\times\BT)^{3\times 3}$ for $r \in [1, 4/3)$, $j = 1, 2, 3$.
If $\kappa_0>0$ such that $|\kappa|\leq \kappa_0$, then 
\begin{equation}
\|\Gamma_\perp^\kappa\|_{\rL_q(\BR^3\times\BT)}
+\|\nabla\Gamma_\perp^\kappa\|_{\rL_r(\BR^3\times\BT)}
\leq C
\label{est:fundsol.int}
\end{equation}
for some constant $C>0$ 
only dependent on $\CT$, $\mu$, $q$, $r$ and $\kappa_0$.
Moreover, for any $\alpha \in \BN_0^3$, $\delta>0$,  $r \in [1, \infty)$
and $\theta>0$ such that $\CT\kappa^2\leq \theta$,
there exists a constant $C > 0$, 
only dependent on $\mu$, $\alpha$, $\delta$, $r$ and $\theta$, such that
\begin{equation}
\forall |x|\geq\delta:\quad
\|\pd_x^\alpha\Gamma_\perp^\kappa(x, \cdot)\|_{\rL_r(\BT)} 
\leq C|x|^{-(3+|\alpha|)}.
\label{est:fundsol.pw}
\end{equation}
\end{thm}

\begin{proof}
The majority of the result was proved in \cite{EK1}. 
However, it was not shown that the respective estimates are uniform
in $\kappa$
if $\kappa$ is small.
To show this property, we reconsider those parts of the proof in \cite{EK1},
where this assumption has an effect.

To study integrability properties of 
$\Gamma_\perp^\kappa$,
the components of $\Gamma_\perp^\kappa$ were expressed as
\[
(\Gamma_\perp^\kappa)_{j\ell}
=\big[
\delta_{j\ell}\mathfrak R_m \mathfrak R_m - \mathfrak R_j \mathfrak R_\ell
\big]
\circ
\sF_{\BR^3\times\BT}^{-1}\big[
M_{\kappa,\CT} \sF_{\BR^3\times\BT}[\Psi]
\big],
\]
where $\mathfrak R_j$ denotes a Riesz transform,
which is a continuous operator on $\rL_{q}(\BR^3\times\BT)$.
Moreover,
$M_{\kappa,\CT}$ is given by
\[
M_{\kappa,\CT}(k,\xi)
:=\frac{(1-\delta_\BZ(k))\,|\frac{2\pi}{\CT}k|^{\frac{2}{5}}(1+|\xi|^2)^{\frac{3}{5}}}{\mu|\xi|^2+i\kappa\xi_1+i\frac{2\pi}{\CT}k},
\]
and 
$\Psi\colon\BR^3\times\BT\to\BR$ is given as the product $\Psi(x,t)=\psi(x)\chi(t)$
with
$\sF_{\BR^3}[\psi](\xi)=(1+|\xi|^2)^{-\frac{3}{5}}$
and $\sF_{\BT}[\chi](k)=(1-\delta_\BZ(k))|\frac{2\pi}{\CT}k|^{-\frac{2}{5}}$.
Then $M_{\kappa,\CT}$ was shown to be an $\rL^{q}$-multiplier 
in $\BR^3\times\BT$
such that
\[
\|\Gamma_\perp^\kappa\|_{\rL_q(\BR^3\times\BT)}
\leq C \|\Psi\|_{\rL_q(\BR^3\times\BT)}.
\]
Going through the proof in~\cite{EK1}, one readily verifies that
the multiplier norm of $M_{\kappa,\CT}$
and thus the constant $C>0$ in this estimate can be chosen uniformly in $\kappa$
if $|\kappa|\leq\kappa_0$.
It was also shown that $\Psi\in\rL^q(\BR^3\times\BT)$ 
for all $q\in(1,5/3)$,
the norm of which is clearly independent of $\kappa$.
This leads to the asserted estimate for $\Gamma_\perp^\kappa$,
and arguing in the same way
for $\nabla\Gamma_\perp^\kappa$,
we obtain the uniform estimate~\eqref{est:fundsol.int}.

To derive the pointwise estimate~\eqref{est:fundsol.pw},
a central term in the proof from~\cite{EK1} given by
\[
\mu(\kappa,k):=
\Big(\frac{\kappa}{2}\Big)^2+i\frac{2\pi}{\CT}k
\]
for $k\in\BZ$,
and its square root $\sqrt{-\mu(\kappa,k)}$,
where $\sqrt{z}$ denotes the square root of $z$ with nonnegative imaginary part.
In particular, we need a 
constant $C_\theta>0$, only depending on $\theta$, such that
\begin{equation}\label{est:impart}
\mathrm{Im} \sqrt{-\mu(\kappa,k)}- \frac{|\kappa|}{2}
\geq C_\theta \big|\frac{2\pi}{\CT}k\big|^\frac{1}{2}
\end{equation}
for all $k\in\BZ\setminus\{0\}$.
Repeating the calculations in \cite[Lemma 3.2]{EK1},
we obtain
\[
\mathrm{Im}\sqrt{-\mu(\kappa,k)}- \frac{|\kappa|}{2}
= \big|\frac{2\pi}{\CT}k\big| \ 
\Phi\Big(\frac{|\kappa|/2}{|\frac{2\pi}{\CT}k|^{1/2}} \Big)
\]
for $\Phi$ given by
\[ 
\Phi(s)
:=s \Big(\frac{1}{\sqrt{2}}\Big(1+\big(1+s^{-4}\big)^\frac{1}{2}\Big)^\frac{1}{2}-1\Big).
\]
Since $\lim_{s\to0} \Phi(s)=\frac{1}{\sqrt{2}}$ and $\Phi(s)>0$ for all $s>0$,
estimate~\eqref{est:impart}
follows with 
\[
C_\theta=\min\big\{\Phi(s)\mid 0<s\leq\frac{\sqrt{\theta}}{2\sqrt{2\pi}}\big\}>0.
\]
Moreover, we have
\[
\big|\frac{2\pi}{\CT}k \big|
\leq 
|\mu(\kappa,k)|
=\big|\frac{2\pi}{\CT}k \big|
\sqrt{1+\Big(\frac{|\kappa|^2/4}{|\frac{2\pi}{\CT}k|} \Big)^2}
\leq \tilde C_\theta\, \big|\frac{2\pi}{\CT}k \big|
\]
with $\tilde C_\theta^2 = 1+(\theta/8\pi)^2$,
so that $\mu$ is comparable with $\big|\frac{2\pi}{\CT}k \big|$
with constants only depending on $\theta$.
Based on these observations, one can now repeat the proof
of the pointwise estimate~\eqref{est:fundsol.pw}
given in~\cite{EK1}
and see that all constants can be chosen uniformly in $\kappa$ and $\CT$ 
as long as $\CT\kappa^2\leq\theta$.
\end{proof}

We can now show the statements of Theorem~\ref{thm:osc.decay}.

\begin{proof}[Proof of Theorem \ref{thm:osc.decay}]
We proceed as in the proof of~\cite[Theorem 5.6]{EiterKyedShibata_PeriodicLp},
where the result was proved for $\kappa=0$.
In order to clarify that the constant $C$ in \eqref{eq:4.2}
can be chosen independently of $\kappa$ for $|\kappa|\leq \kappa_0$,
we repeat the arguments here.

Since we assume $3 < q < \infty$, Sobolev embeddings
and estimate~\eqref{est:osc}
imply
$$\sup_{|x| \leq 4b} \|\bv_\perp(\cdot, x)\|_{\rL_p(\BT)} 
+ \sup_{|x| \leq 4b} \|(\nabla \bv_\perp)(\cdot, x)\|_{\rL_p(\BT)} 
\leq C\|\bv_\perp\|_{\rL_p(\BT, \rH^2_q(\Omega))}
\leq C\|\bff_\perp\|_{\rL_p(\BT, \rL_q(\Omega))}.
$$
In virtue of Remark~\ref{rem:osc.decay},
it thus remains to estimate $\bv_\perp$ for $|x| > 4b$. 
As seen in the proof of Theorem \ref{thm:osc}, 
we have $\bv_\perp = \sF^{-1}_\BT[\sS_\kappa(i\tfrac{2\pi}{\CT}k)\sF_\BT[\bff_\perp](k)]$ and $\fp_\perp
= \sF^{-1}_\BT[\sP_\kappa(i\tfrac{2\pi}{\CT}k)\sF_\BT[\bff_\perp](k)]$,
where $\sS_\kappa$ and $\sP_\kappa$ are the families of solution operators given in Theorem \ref{thm:res}.

We first derive a representation formula of 
$\sS_\kappa(i\tfrac{2\pi}{\CT}k)$ for $k \in \BZ\setminus\{0\}$ and $|x| > 4b$.
Notice that 
$\sS_\kappa(i\tfrac{2\pi}{\CT}k) \in \sL(\rL_q(\Omega)^3, \rH^2_q(\Omega)^3)$ and 
$\sP_\kappa(i\tfrac{2\pi}{\CT}k) \in \sL(\rL_q(\Omega)^3, \hat \rH^1_q(\Omega))$ 
satisfy
the estimate
\begin{equation}\label{est:6.4}
 \|\sS_\kappa(i\tfrac{2\pi}{\CT}k)\bff\|_{\rH^2_q(\Omega)} 
 + \|\nabla \sP_\kappa(i\tfrac{2\pi}{\CT}k) \bff
\|_{\rL_q(\Omega)} \leq C\|\bff\|_{\rL_q(\Omega)}
\end{equation}
for $\bff\in\rL_q(\Omega)^3$,
where 
$C$ depends solely on $\Omega$, $\mu$, $q$ and $\kappa_0$. 
Moreover, the functions $\bu=\sF_\BT[\bv_\perp](k)=\sS_\kappa(i\tfrac{2\pi}{\CT}k)\sF_\BT[\bff_\perp](k)$ and $\fq = 
\sF_\BT[\fp_\perp](k)=\sP_\kappa(i\tfrac{2\pi}{\CT}k)\sF_\BT[\bff_\perp](k)$ satisfy the equations
\begin{equation}\label{eq:perp.1}
i\tfrac{2\pi}{\CT}k\bu - \mu\Delta\bu - \kappa\partial_1\bu + \nabla \fq = \bff_k, 
\quad \dv\bu= 0 \quad\text{in $\Omega$},
\quad \bu|_\Gamma=0,
\end{equation}
where $\bff_k = \sF_\BT[\bff_\perp](k)$.
Let $\varphi$ be a function in $C^\infty_0(\BR^3)$ that equals $1$ for $|x| < 2b$
and $0$ for $|x| > 3b$.  Let  
\begin{equation}\label{proof:4.2}
\bw= (1-\varphi)\sS_\kappa(i\tfrac{2\pi}{\CT}k)\bff_k + \BB[(\nabla\varphi)\cdot
\sS_\kappa(i\tfrac{2\pi}{\CT}k)\bff_k], \quad 
\fr = (1-\varphi)\sP_\kappa(i\tfrac{2\pi}{\CT}k)\bff_k,
\end{equation}
where $\BB$ denotes the Bogovski\u\i{} operator \cite{Bogovskii79,Bogovskii80}.
Then $\bw \in \rH^2_q(\BR^3)^3$ and $\fr \in \hat \rH^1_q(\BR^3)$,
and the functions  $\bw$ and $\fr$ satisfy the equations
$$
i\tfrac{2\pi}{\CT}k\bw - \mu\Delta\bw - \kappa\partial_1\bw + \nabla \fr = 
(1-\varphi)\bff_k + \CR_{1,\kappa}(i\tfrac{2\pi}{\CT}k)\bff_k, \quad \dv \bw = 0
\quad\text{in $\BR^3$},
$$
where we have set 
\[
\begin{aligned}
\CR_{1,\kappa}(\lambda)\bff 
&= 2\mu(\nabla\varphi)\cdot\nabla \sS_\kappa(\lambda)\bff
+(\mu\Delta\varphi+\kappa\partial_1\varphi)\sS_\kappa(\lambda)\bff - (\nabla \varphi)\sP_\kappa(\lambda)\bff
\\
&\qquad
+(\lambda-\mu\Delta-\kappa\partial_1)\BB[(\nabla\varphi)\cdot
\sS_\kappa(\lambda)\bff].
\end{aligned}
\]
By the uniqueness of solutions to the Oseen resolvent problem in $\BR^3$, we have
$\bw = \CT_\kappa(i\tfrac{2\pi}{\CT}k)((1-\varphi)\bff_k + \CR_{1,\kappa}(i\tfrac{2\pi}{\CT}k)\bff_k)$,
where
\begin{equation}\label{eq:2.3}
\CT_\kappa(\lambda)\bff = \sF^{-1}_{\BR^3}\Bigl[\frac{1}
{\mu|\xi|^2 - i\kappa\xi_1 + \lambda}\Bigr({\rm I} - \frac{\xi\otimes\xi}{|\xi|^2}\Bigr)\sF_{\BR^3}[\bff]\Bigr].
\end{equation}
Since $1-\varphi(x) = 1$ 
and $\BB[(\nabla\varphi)\cdot\sS_\kappa(i\tfrac{2\pi}{\CT}k)\bff_k](x,t)=0$ 
for $|x| >4b$, by \eqref{proof:4.2} we thus 
obtain
\[
\sS_\kappa(i\tfrac{2\pi}{\CT}k)\bff_k(x) = \CT_\kappa(i\tfrac{2\pi}{\CT}k)((1-\varphi)\bff_k)(x)
+ \CT_\kappa(i\tfrac{2\pi}{\CT}k)(\CR_{1,\kappa}(i\tfrac{2\pi}{\CT}k)\bff_k)(x)
\]
for $|x| > 4b$ and any $k \in \BZ\setminus\{0\}$.

This representation formula implies
\begin{equation}\label{eq:perp.2}\begin{aligned}
\bv_\perp &= \sF^{-1}_\BT[(1-\delta_\BZ(k))\sS_\kappa(i\tfrac{2\pi}{\CT}k)\sF_\BT[\bff_\perp](k)]
\\
&= \sF^{-1}_\BT[(1-\delta_\BZ(k))\CT_\kappa(i\tfrac{2\pi}{\CT}k)\sF_\BT[
(1-\varphi)\bff_\perp](k))]\\
&\qquad
+ \sF^{-1}_\BT[(1-\delta_\BZ(k))\CT_\kappa(i\tfrac{2\pi}{\CT}k)(\CR_{1,\kappa}(i\tfrac{2\pi}{\CT}k)
\sF_\BT[\bff_\perp](k))]
\end{aligned}\end{equation} 
for $|x| > 4b$. 
Moreover, from Theorem \ref{thm:res} we conclude 
\[
\begin{aligned}
&\sR_{\sL(\rL_q(\Omega)^3, \rH^1_q(\BR^3)^3)}(\{(\lambda\pd_\lambda)^\ell
\CR_{1,\kappa}(\lambda) \mid 
\lambda \in \BR\setminus[-\lambda_0, \lambda_0]\}) \leq r_0
\quad(\ell=0,1), \\
&\| \CR_{1,\kappa}(i\tfrac{2\pi}{\CT}k)\bff_k\|_{\rH^1_q(\BR^3)} \leq r_0\|\bff_k\|_{\rL_q(\Omega)}
\end{aligned}
\]
for any  $k\in\BZ\setminus\{0\}$  
with some constant 
 $r_0$ independent of $\kappa$. 
We can thus define $\CR_{2,\kappa}\bff_\perp$ by setting 
$\CR_{2,\kappa}\bff_\perp= \sF^{-1}_\BT[
(1-\delta_\BZ(k))\CR_{1,\kappa}(i\tfrac{2\pi}{\CT}k)\bff_k]$
and obtain that 
\begin{equation}\label{eq:StokesRes_extdom4}\begin{aligned}
&{\rm supp}\, \CR_{2,\kappa}\bff_\perp \subset D_{2b, 3b} 
: = \{(x, t) \in \BR^3 \times \BR \mid 2b < |x| < 3b\}, \\
&\qquad\|\CR_{2,\kappa}\bff_\perp\|_{\rL_p(\BT, \rL_q(\Omega))} 
\leq C\|\bff_\perp\|_{\rL_p(\BT, \rL_q(\Omega))}
\end{aligned}
\end{equation}
by employing Theorem~\ref{thm:multiplier}
in the same way as in the proof of Theorem \ref{thm:osc}. 
Recalling that $\bff_\perp = \dv\bF_\perp + \bg_\perp$, we set 
 $\bG = (1-\varphi)\bF_\perp$
and $\bh =  
(\nabla\varphi)\bF_\perp + (1-\varphi)\bg_\perp +
\CR_{2,\kappa}\bff_\perp$.  
In virtue of \eqref{kernel}, \eqref{eq:2.3} and \eqref{eq:perp.2}, we then have
\[
\bv_\perp(x,t) = \Gamma_\perp^\kappa*(\dv\bG)(x,t) + \Gamma_\perp^\kappa*\bh(x,t) 
\]
for $t\in\BT$ and $|x| > 4b$.  

We decompose this formula into two parts 
and set $\bv_1 = \Gamma_\perp^\kappa*(\dv\bG)$ and
$\bv_2 = \Gamma_\perp^\kappa*\bh$. 
By the divergence theorem,   we obtain
\begin{align*}
&\bv_1(x,t) = \nabla \Gamma_\perp^\kappa * \bG(x,t)\\
 & = \int_\BT\int_{|y| \leq 1}\nabla\Gamma_\perp^\kappa(y, s)
\bG(x-y, t-s)\,\dd y\dd s + \int_\BT\int_{1 \leq |y| \leq |x|/2}
\nabla\Gamma_\perp^\kappa(y, s)\bG(x-y, t-s)\,\dd y\dd s \\
&+\int_\BT\int_{|x|/2\leq |y| \leq 2|x|} \nabla\Gamma_\perp^\kappa(y,s)
\bG(x-y, t-s)\,\dd y\dd s 
+\int_\BT\int_{ |y| \geq 2|x|} \nabla\Gamma_\perp^\kappa(y,s)
\bG(x-y, t-s)\,\dd y\dd s.
\end{align*}
We set $\gamma_\ell = <\bG>_{p,\ell}$ and consider $p < r_0 < \infty$, $r_1 \in (1, 5/4)$
such that $1 + 1/r_0 = 1/r_1 + 1/p$.
From Young's inequality  and Theorem \ref{lem:4.1}, we thus conclude 
\begin{align*}
\|\bv_1(x, \cdot)\|_{\rL_{r_0}(\BT)} 
&\leq \gamma_\ell\|\nabla\Gamma_\perp^\kappa\|_{\rL_{r_1}(B_1\times \BT)}(1+ |x|)^{-\ell}
+C\gamma_\ell(1+|x|)^{-\ell}\int_{1\leq |y| \leq |x|/2}|y|^{-4}\,\dd y \\
& \qquad+ C\gamma_\ell(|x|/2)^{-4}\int_{|z| \leq 3|x|}(1 + |z|)^{-\ell}\,\dd z 
+ C\gamma_\ell\int_{|y| \geq 2|x|}|y|^{-4-\ell}\,\dd y.
\end{align*} 
Noting that $p \leq r_0$ and $\gamma_\ell \leq \ <\bF_\perp>_{p, \ell}$, we infer 
\begin{align*}\| \bv_1(x, \cdot)\|_{\rL_p(\BT)} &\leq C_b|x|^{-\min\{\ell,4\}}<\bF_\perp>_{p, \ell}
\quad\text{for $|x| \geq 4b$}.
\end{align*}
In the same way we 
decompose $\nabla \bv_1=\nabla\Gamma_\perp^\kappa\ast\dv \bG$
and use
Theorem \ref{lem:4.1} and Young's inequality to obtain 
\begin{align*}
\|\nabla\bv_1(x, \cdot)\|_{\rL_{r_0}(\BT)}
&\leq \gamma_{\ell+1}\|\nabla\Gamma_\ell\|_{\rL_{r_1}(B_1\times\BT)}(1+ |x|)^{-\ell-1}
+C\gamma_{\ell+1}(1+ |x|)^{-\ell-1}\int_{1\leq |y| \leq |x|/2}|y|^{-4}\,\dd y \\
&\qquad + C\gamma_{\ell+1}(|x|/2)^{-4}\int_{|z| \leq 3|x|}(1 + |z|)^{-\ell-1}\,\dd z 
+ C\gamma_{\ell+1}\int_{|y| \geq 2|x|}|y|^{-5-\ell}\,\dd y,
\end{align*}
where 
$\gamma_{\ell+1} = <\dv\bG>_{p, \ell+1}$.
Since we have
$$
\begin{aligned}
<\dv\bG>_{p, \ell+1} 
&\leq\  <\dv \bF_\perp>_{p, \ell+1} + <(\nabla\varphi)
\bF_\perp>_{p, \ell+1}
\\
&\leq\  <\dv \bF>_{p, \ell+1} + \|\nabla\varphi\|_{\rL_\infty(\BR^3)}
3b<\bF>_{p, \ell}
\end{aligned}
$$
and $p \leq r_0$, we thus obtain 
\begin{align*}\|\nabla\bv_1(x, \cdot)\|_{\rL_p(\BT)} &\leq C_b|x|^{-\min\{\ell+1,4\}}
(<\dv\bF_\perp>_{p, \ell+1} + <\bF_\perp>_{p, \ell})
\quad \text{for $|x| \geq 4b$}. 
\end{align*}
Using that $\bh(y,s)$ vanishes for $|y|\geq 3b$,
we obtain have
\[
\nabla^m\bv_2(x,t)
= \int_\BT\int_{|x-y| \leq 3b}\nabla^m\Gamma_\perp^\kappa(y, s)
\bh(x-y, t-s)\,\dd y\dd s
\]
for $m=0,1$.
Since $|x|\geq 4b$ and $|x-y|\leq 3b$ implies $|y|\geq |x|/4\geq b$,
by Theorem \ref{lem:4.1} and Young's inequality, we deduce 
\[
\begin{aligned}
\|\nabla^m\bv_2(x, \cdot)\|_{\rL_{p}(\BT)}
&\leq  \int_{|x-y| \leq 3b}\|\nabla^m\Gamma_\perp^\kappa(y,\cdot)\|_{\rL_p(\BT)}
\|\bh(x-y, \cdot)\|_{\rL_1(\BT)}\,\dd y
\\
&\leq  C_m |x|^{-3-m}\|\bh\|_{\rL_1(B_{3b}\times\BT)}.
\end{aligned}
\]
Noting \eqref{eq:StokesRes_extdom4},
we can estimate the last term as
\[
\begin{aligned}
\|\bh\|_{\rL_1(B_{3b}\times\BT)}
\leq
C\|\bh\|_{\rL_p(\BT,\rL_q(B_{3b}))}
&\leq
C\big(
<\bF_\perp>_{p, \ell} + \|\bg_\perp\|_{\rL_p(\BT, \rL_q(\Omega))}
+ \|\CR_{2,\kappa}\bff_\perp\|_{\rL_p(\BT, \rL_q(\Omega))}\big) \\
&\leq C\big(
<\bF_\perp>_{p, \ell} + \|\bg_\perp\|_{\rL_p(\BT, \rL_q(\Omega))}
+ \|\bff_\perp\|_{\rL_p(\BT, \rL_q(\Omega))}\big).
\end{aligned}
\]
For $|x|\geq 4b$ we now conclude
\[
\|\nabla^m\bv_2(x, \cdot)\|_{\rL_{p}(\BT)}
\leq C|x|^{-3-m}\big(
<\bF_\perp>_{p, \ell} + \|\bg_\perp\|_{\rL_p(\BT, \rL_q(\Omega))}
+ <\dv\bF_\perp>_{p, \ell+1}\big)
\]
in virtue of Remark~\ref{rem:osc.decay}.
Since $\bv(x,t)=\bv_1(x,t)+\bv_2(x,t)$ for $|x|\geq 4b$, we conclude~\eqref{eq:4.2}.
\end{proof}

\subsection{The full time-periodic problem}

To conclude the 
existence result for time-periodic solutions to the Oseen problem
as stated in Theorem~\ref{thm:tpOseen},
we combine  
Theorem \ref{thm:stat} and Theorem \ref{thm:osc.decay}
with a lifting argument for inhomogeneous boundary conditions.

\begin{proof}[Proof of Theorem \ref{thm:tpOseen}]
We first reduce the problem to the case of homogeneous boundary conditions.
For this purpose, let 
$$\bv_1 \in \rH^1_p(\BT, \rL_q(\Omega)^3) \cap 
\rL_p(\BT, \rH^2_q(\Omega)^3), \quad
\fp_1 \in \rL_p(\BT, \hat\rH^1_q(\Omega)),
$$
be a solution to the time-periodic Stokes problem
\[
\pd_t\bv_1 - \mu\Delta\bv_1 + \nabla\fp_1 
= 0,
\quad \dv\bv_1 = 0\quad 
\text{in $\Omega\times\BT$}, \quad 
\bv_1|_{\Gamma\times\BT}  = \bh|_{\Gamma\times\BT},
\]
which exists due to Theorem~\ref{thm:tpStokes}.
Let $\varphi\in C^\infty_0(\Omega)$ be
such that $\varphi\equiv 1$ in $B_{2b}$ and $\varphi\equiv 0$ in $\BR^3\setminus B_{3b}$. 
Let $D_{2b, 3b} = \{x \in \BR^3 \mid 2b < |x| < 3b\}$ and 
$$
\rH^2_{q, 0, a}(D_{2b, 3b}) =\big\{f \in \rH^2_q(D_{2b, 3b}) \mid \pd_x^\alpha f|_{S_L}=0
\enskip \text{for $L=2b$, $3b$ and $|\alpha|\leq 1$}, \enskip 
\int_{D_{2b, 3b}} f(x)\,\dd x = 0\big\}.
$$
According to \cite[Lemma 5]{Shibata18}, we know that 
$(\nabla\varphi)\cdot\bv_1(t) \in \rH^2_{q, 0, a}(D_{2b, 3b})$ for a.a.~$t\in\BR$,
and setting
$\tilde\bv = \varphi \bv_1 - \BB[(\nabla\varphi)\cdot\bv_1]$, we see that
\begin{equation}
\begin{aligned}
&\tilde\bv \in \rH^1_p(\BT, \rL_q(\Omega)^3) \cap \rL_p(\BT, \rH^2_q(\Omega)^3), 
\quad {\rm supp}\,\tilde\bv \subset B_{3b} \cap \overline\Omega, 
\quad \dv\tilde\bv=0,
\quad \tilde\bv|_\Gamma = \bh,\\
&\|\pd_t\tilde\bv\|_{\rL_p(\BT, \rL_q(\Omega))}
+ \|\tilde\bv\|_{\rL_p(\BT, \rH^2_q(\Omega))}
\leq C(\|\pd_t\bh\|_{\rL_p(\BT, \rL_q(\Omega))} + \|\bh\|_{\rL_p(\BT, \rH^2_q(\Omega))}).
\end{aligned}
\label{est:lift}
\end{equation}
Now let $(\bw_S,\fq_S)$ and $(\bw_\bot,\fq_\bot)$ be the unique solutions to
the equations
\begin{align}
- \Delta\bw_S - \kappa\partial_1\bw_S + \nabla \fq_S 
&=\bff_S+\kappa\partial_1\tilde\bv_S,
& \dv \bw_S &= 0
\quad\text{in
$\Omega$}, &
\bw_S|_\Gamma&=0,
\label{eq:oseen.split1}
\\
\pd_t\bw_\bot - \Delta\bw_\bot - \kappa\partial_1\bw_\bot + \nabla \fq_\bot 
&=\bff_\bot+\kappa\partial_1\tilde\bv_\bot,
& \dv \bw_\bot &= 0
\quad\text{in
$\Omega\times\BT$}, &
\bw_\bot|_{\Gamma\times\BT}&=0,
\label{eq:oseen.split2}
\end{align}
which exist due to Theorem~\ref{thm:stat} and Theorem~\ref{thm:osc}.
Also invoking Theorem~\ref{thm:osc.decay},
we see that 
$(\bv,\fp)=(\tilde\bv+\bw_S+\bw_\bot,\fq_S+\fq_\bot)$
is a solution to the original problem \eqref{eq:nslin.tp}
and belongs to the asserted function class since 
$\tilde\bv$ vanishes in $\BR^3\setminus B_{3b}$.
Estimate~\eqref{est:tpOseen}
follows
by combining the estimates from~\eqref{est:statOseen},
\eqref{est:osc}, \eqref{eq:4.2} and~\eqref{est:lift}.

The uniqueness statement is a direct consequence 
of decomposing a solution $(\bv,\fp)$
into a stationary and an oscillatory part by means of \eqref{eq:decomposition}
and using the respective statements from 
Theorem \ref{thm:stat} and Theorem \ref{thm:osc}.
\end{proof}

\begin{proof}[Proof of Theorem~\ref{thm:tpOseen.int}]
We proceed as in the proof of Theorem~\ref{thm:tpOseen}
and first reduce the problem to the case of homogeneous boundary conditions.
We construct the function $\tilde\bv$ as before,
and let $(\bw_S,\fq_S)$ and $(\bw_\bot,\fq_\bot)$
be the solutions to~\eqref{eq:oseen.split1} and~\eqref{eq:oseen.split2},
which exist due to
Theorem~\ref{thm:stat.Lq}
and Theorem~\ref{thm:osc}.
Then $(\bv,\fp)=(\tilde\bv+\bw_S+\bw_\bot,\fq_S+\fq_\bot)$
is a solution to~\eqref{eq:nslin.tp}
and the estimates~\eqref{est:tpOseen.int} and~\eqref{est:tpOseen.int2}
follow from~\eqref{est:statOseen.Lq},~\eqref{est:osc} and~\eqref{est:lift}.
As before, the uniqueness statement follows from
the respective statements of Theorem~\ref{thm:stat.Lq} and Theorem~\ref{thm:osc}.
\end{proof}

\section{Existence of solutions to the nonlinear problem}
\label{sec:nonlin}

To show existence of time-periodic solutions 
to the nonlinear problem~\eqref{eq:systemref},
we combine the linear theory
from Theorem~\ref{thm:tpStokes}
Theorem~\ref{thm:tpOseen} and Theorem~\ref{thm:tpOseen.int}
with suitable estimates for the linear perturbation term $\CL$ 
and the nonlinear term $\CN$
given in \eqref{6.6}.
Then 
Banach's fixed-point theorem
will finally lead 
to the proofs of Theorem~\ref{mainthm:Stokes}, Theorem~\ref{mainthm:Oseen}
and Theorem~\ref{mainthm:Oseen.int}.
More precisely, we consider the set
\begin{equation}
\CI_{\kappa,\rho} = \big\{(\bv, \fq) \mid 
\bv\in \rL_{1,\mathrm{loc}}(\Omega\times\BT)^3,  
\quad  
 \fq \in  \rL_{1,\mathrm{loc}}(\Omega\times\BT), \quad
\|(\bv, \fq)\|_{\CI_\kappa} \leq \rho\big\}
\label{eq:fpset}
\end{equation}
for a suitable norm $\|\cdot\|_{\CI_{\kappa}}$
that is defined in the respective proofs
and is suggested by the associated linear theory.
For given $(\bv, \fq) \in \CI_{\kappa, \rho}$,
we then consider the solution $(\bu, \fp)$ to the linear system
\begin{equation}
\pd_t\bu - \mu\Delta\bu - \kappa\pd_1\bu + \nabla\fp = \bff + \CL(\bv, \fq) + \CN(\bv),
\quad\dv\bu=0
\quad\text{in $\Omega\times\BT$}, \qquad \bu|_{\Gamma\times\BT}=\bh|_{\Gamma\times\BT},
\label{eq:fpsystem}
\end{equation}
where $\CL$ and $\CN$ are defined in~\eqref{6.6}.
We show that in the respective settings and 
for a suitable choice of $\rho$,
this leads to a well-defined solution mapping
$\Xi_\kappa\colon\CI_{\kappa,\rho}\to\CI_{\kappa,\rho}$,
$(\bv,\fq)\mapsto(\bu,\fp)$,
which is contractive. 
Therefore, Banach's fixed-point theorem yields
the existence of an element $(\bw,\fq)$
such that 
$\Xi_\kappa(\bw,\fq)=(\bw,\fq)$,
that is,
$(\bw,\fq)$ is a solution to the nonlinear problem \eqref{eq:systemref}.

To derive suitable estimates,
we write $\CL(\bv, \fq)$ and $\CN(\bv)$ as 
\begin{equation}\label{repr:el1}\begin{aligned}
\CL(\bv, \fq) &= \CL_1\pd_t\bv + \sum_{|\alpha| \leq 2}
\CL_{2, \alpha}D^\alpha\bv + \sum_{|\alpha|=1}\CL_{3, \alpha}D^\alpha \fq,\\
\CN(\bv) & = \CN^1(\bv) + \CN^2(\bv) = 
\bv\cdot\nabla\bv + \sum_{|\alpha| \leq 1} \CL_{4, \alpha}\bv\cdot
D^\alpha\bv.
\end{aligned}\end{equation}
Here, $\CL_1$, $\CL_{2, \alpha}$, $\CL_{3, \alpha}$ and 
$\CL_{4, \alpha}$ correspond
to time-periodic continuous functions on 
$\Omega\times\BT$, which vanish on $\BR^3\setminus B_{2b}\times\BT$ and satisfy the estimate
\begin{equation}\label{repr:el2}
\|(\CL_1, \CL_{2, \alpha}, \CL_{3, \alpha}, \CL_{4, \alpha})
\|_{\rL_\infty(\Omega\times\BT)} \leq C\varepsilon_0
\end{equation}
if $|\kappa|\leq\kappa_0$,
due to~\eqref{5.4} and \eqref{5.4*}.
We define $\Omega_{2b}=\Omega\cap B_{2b}$.
The terms that vanish outside of $\Omega_{2b}\times\BT$,
where $\Omega_{2b}=\Omega\cap B_{2b}$,
can be estimated in the following manner.

\begin{lem}\label{lem:rhs.loc}
Let $2<p<\infty$ and $3<q<\infty$, 
and let
\[
\bv\in \rH^1_p(\BT, \rL_q(\Omega_{2b})^3) \cap 
\rL_p(\BT, \rH^2_q(\Omega_{2b})^3),
\qquad
\nabla\fq\in\rL_p(\BT, \rL_q(\Omega_{2b})^3).
\]
It holds
\begin{align}
&\|\CL(\bv, \fq)\|_{\rL_p(\BT, \rL_q(\Omega))}
\leq C\varepsilon_0\big(\|\pd_t\bv\|_{\rL_p(\BT, \rL_q(\Omega_{2b}))}
+ \|\bv\|_{\rL_p(\BT, \rH^2_q(\Omega_{2b}))} 
+ \|\nabla\fq\|_{\rL_p(\BT, \rL_q(\Omega_{2b}))} \big),
\label{est:rhs.loclin}\\
&\|\CN^2(\bv)\|_{\rL_p(\BT, \rL_q(\Omega))} 
\leq C\varepsilon_0
\big(\|\pd_t\bv\|_{\rL_p(\BT, \rL_q(\Omega_{2b}))}
+ \|\bv\|_{\rL_p(\BT, \rH^2_q(\Omega_{2b}))}\big)^2.
\label{est:rhs.locquad}
\end{align}
\end{lem}
\begin{proof}
Estimate~\eqref{est:rhs.loclin} is a direct consequence of \eqref{repr:el1}
and \eqref{repr:el2}.
For estimate~\eqref{est:rhs.locquad},
we first use H\"older's inequality and \eqref{repr:el2} to obtain
\[
\|\CN^2(\bv)_\perp\|_{\rL_p(\BT, \rL_q(\Omega))} 
\leq C\varepsilon_0\|\bv\|_{\rL_\infty(\BT, \rL_\infty(\Omega_{2b}))}
\|\nabla\bv\|_{\rL_p(\BT, \rL_q(\Omega_{2b}))}
\]
Now choose $\sigma>0$ so small that $\sigma + 3/q < 2(1-1/p)$,
which is possible due to $2/p + 3/q < 2$.
Using Sobolev's inequality and real interpolation,  
we then have
\begin{equation}\label{est:interpolation}\begin{aligned}
\|\bw\|_{\rL_\infty(\BT, \rL_\infty(D))} 
&\leq C\|\bw\|_{\rL_\infty(\BT, \rW^{\sigma + 3/q}_q(D))} 
\leq C\|\bw\|_{\rL_\infty(\BT, \rB^{2(1-1/p)}_{q,p}(D))} \\
&\leq C(\|\pd_t\bw\|_{\rL_p(\BT, \rL_q(D))}
+ \|\bw\|_{\rL_p(\BT, \rH^2_q(D))})
\end{aligned}\end{equation}
for any Lipschitz domain $D\subset\BR^3$.
Using this estimate with $\bw=\bv$ and $D=\Omega_R$ leads to estimate~\eqref{est:rhs.locquad}.
\end{proof}

We shall use these estimates of the local terms in all proofs below.
However, for estimates of the term $\CN^1(\bv)$,
which also gives a contribution far away from the boundary,
the spatial decay of solutions has to be
taken into account.
Observe that $\dv\bv=0$ implies
$\CN^1(\bv)=\dv\tilde\CN^1(\bv)$ with $\tilde\CN^1(\bv) = \bv\otimes\bv$.
Then 
\begin{equation}\label{nonlinear:1}\begin{aligned}
\CN^1(\bv)_S &= \bv_S\cdot\nabla\bv_S + \int_\BT
\bv_\perp\cdot\nabla\bv_\perp\,\dd t\\
\tilde\CN^1(\bv)_S &= \bv_S\otimes\bv_S + \int_\BT
\bv_\perp\otimes\bv_\perp\,\dd t; \\
\CN^1(\bv)_\perp& = \bv_S\cdot\nabla\bv_\perp + \bv_\perp\cdot\nabla\bv_S
+ \bv_\perp\cdot\nabla\bv_\perp - \int_\BT
\bv_\perp\cdot\nabla\bv_\perp\,\dd t \\
\tilde\CN^1(\bv)_\perp &= \bv_S\otimes\bv_\perp + \bv_\perp\otimes\bv_S
+ \bv_\perp\otimes\bv_\perp - \int_\BT
\bv_\perp\otimes\bv_\perp\,\dd t.
\end{aligned}\end{equation}
and $\dv\tilde\CN(\bv)_S = \CN^1(\bv)_S$ and $\dv\tilde\CN(\bv)_\perp
=\CN^1(\bv)_\perp$. 
Corresponding estimates that suit to the linear theory
are derived in the following three lemmas.

\begin{lem}\label{lem:nonlinest.Stokes}
Let $2<p<\infty$ and $3<q<\infty$, 
and let
\[
\bv\in \rH^1_p(\BT, \rL_q(\Omega)^3) \cap 
\rL_p(\BT, \rH^2_q(\Omega)^3),
\qquad
<\bv>_{p,1} + <\nabla\bv>_{p, 2} 
<\infty.
\]
Then
\[
\begin{aligned}
<\CN^1(\bv)_S>_3
&\leq C(<\bv_S>_1<\nabla\bv_S>_2 
+ <\bv_\perp>_{p, 1}<\nabla\bv_\perp>_{p, 2}),
\\
<\tilde\CN^1(\bv)_S>_2
&\leq C(<\bv_S>_1^2
+ <\bv_\perp>_{p, 1}^2),
\\
<\CN^1(\bv)_\perp>_{p, 2} &\leq C\big( <\bv_\perp>_{p, 1}
<\nabla\bv_S>_2
\\
&\qquad\quad
+ (\|\pd_t\bv\|_{\rL_p(\BT, \rL_q(\Omega))}
+ \|\bv\|_{\rL_p(\BT, \rH^2_q(\Omega))}
+<\bv>_{p, 1})
<\nabla\bv_\perp>_{p, 2}\big),
\\
<\tilde\CN^1(\bv)_\perp>_{p, 1} & \leq C\big(
\|\pd_t\bv\|_{\rL_p(\BT, \rL_q(\Omega))}
+ \|\bv\|_{\rL_p(\BT, \rH^2_q(\Omega))}
+<\bv>_{p, 1}\big)
<\bv_\perp>_{p, 1}.
\end{aligned}
\]
\end{lem}

\begin{proof}
The estimates of $\CN(\bv)_S$ and $\tilde\CN(\bv)_S$
follow directly
from using H\"older's inequality for the time integrals since $p>2$.
For the estimates of $\CN(\bv)_\perp$ and $\tilde\CN(\bv)_\perp$,
H\"older's inequality leads to
\[
\begin{aligned}
<\CN^1(\bv)_\perp>_{p, 2} &\leq C(<\bv_S>_{1}
<\nabla\bv_\perp>_{p, 2} + <\bv_\perp>_{p, 1}
<\nabla\bv_S>_2\\
&\qquad+ \|\bv_\perp\|_{\rL_\infty(\BT, \rL_\infty(\Omega))}<\nabla\bv_\perp>_{p, 2}
+<\bv_\perp>_{p, 1}<\nabla\bv_\perp>_{p, 2}),
\\
<\tilde\CN^1(\bv)_\perp>_{p, 1} & \leq C(<\bv_S>_1
<\bv_\perp>_{p, 1} + \|\bv_\perp\|_{\rL_\infty(\BT, \rL_\infty(\Omega))}
<\bv_\perp>_{p, 1} + <\bv_\perp>^2_{p, 1}).
\end{aligned}
\]
Using now the interpolation inequality~\eqref{est:interpolation}
completes the proof.
\end{proof}

\begin{lem}\label{lem:nonlinest.Oseen}
Let $2<p<\infty$, $3<q<\infty$ and $\delta\in(0,\frac{1}{4})$,
and let
\[
\begin{aligned}
&\bv\in \rH^1_p(\BT, \rL_q(\Omega)^3) \cap 
\rL_p(\BT, \rH^2_q(\Omega)^3),
\\
&<\bv_S>^w_{1, \delta} + <\nabla \bv_S>^w_{3/2, 1/2+\delta}
+ <\bv_\perp>_{p,1+\delta} + <\nabla\bv_\perp>_{p, 2+\delta}
<\infty.
\end{aligned}
\]
Then
\[
\begin{aligned}
&<\CN^1(\bv)_S>^w_{5/2,1/2+2\delta} 
\leq C\big(<\bv_S>^w_{1,\delta}<\nabla\bv_S>^w_{3/2, 1/2+\delta}
+ <\bv_\perp>_{p, 1+\delta}<\nabla\bv_\perp>_{p, 2+\delta}\big), 
\\
&<\CN^1(\bv)_\perp>_{p, 2+\delta} 
\leq C \big( <\bv_\perp>_{p, 1+\delta}
<\nabla\bv_S>^w_{3/2, 1/2+\delta} 
\\
&\qquad\qquad
+(\|\pd_t\bv\|_{\rL_p(\BT, \rL_q(\Omega))}
+ \|\bv\|_{\rL_p(\BT, \rH^2_q(\Omega))}
+<\bv_S>^w_{1,\delta}+<\bv_\perp>_{p, 1+\delta})
<\nabla\bv_\perp>_{p, 2+\delta}\big ),
\\
&<\tilde\CN^1(\bv)_\perp>_{p, 1+\delta} 
\\
&\qquad\qquad
 \leq C(\|\pd_t\bv\|_{\rL_p(\BT, \rL_q(\Omega))}
+ \|\bv\|_{\rL_p(\BT, \rH^2_q(\Omega))}
+<\bv_S>^w_{1,\delta}+<\bv_\perp>_{p, 1+\delta})
<\bv_\perp>_{p, 1+\delta} .
\end{aligned}
\]
\end{lem}

\begin{proof}
As in the previous proof, 
the estimate of $\CN^1(\bv)_S$ follows directly from H\"older's inequality.
Moreover, we obtain
\[
\begin{aligned}
<\CN^1(\bv)_\perp>_{p, 2+\delta} 
&\leq C(<\bv_S>^w_{1,\delta}
<\nabla\bv_\perp>_{p, 2+\delta} + <\bv_\perp>_{p, 1+\delta}
<\nabla\bv_S>^w_{3/2, 1/2+\delta} 
\\
&\qquad\quad
+ \|\bv_\perp\|_{\rL_\infty(\BT, \rL_\infty(\Omega))}<\nabla\bv_\perp>_{p, 2+\delta}
+<\bv_\perp>_{p, 1+\delta}<\nabla\bv_\perp>_{p, 2+\delta}), \\
<\tilde\CN^1(\bv)_\perp>_{p, 1+\delta} & \leq C(<\bv_S>^w_{1,\delta}
<\bv_\perp>_{p, 1+\delta} + \|\bv_\perp\|_{\rL_\infty(\BT, \rL_\infty(\Omega))}
<\bv_\perp>_{p, 1+\delta} + <\bv_\perp>^2_{p, 1+\delta}).
\end{aligned}
\]
The asserted estimates now result from the 
interpolation inequality~\eqref{est:interpolation}.
\end{proof}

\begin{lem}\label{lem:nonlinest.Oseen.int}
Let $2<p<\infty$, $3<q<\infty$ and $1<s\leq 4/3$,
and let $\bv=\bv_S+\bv_\perp$ be such that
\[
\begin{aligned}
&\bv_S\in\rL_{2s/(2-s)}(\Omega), \quad 
\nabla\bv_S\in\rL_{4s/(4-s)}(\Omega), \quad
\nabla^2\bv_S\in\rL_s(\Omega)\cap\rL_q(\Omega),
\\
&\bv_\perp\in \rH^1_p(\BT, \rL_q(\Omega)^3\cap\rL_s(\Omega)^3) \cap 
\rL_p(\BT, \rH^2_q(\Omega)^3\cap\rH^2_s(\Omega)^3).
\end{aligned}
\]
Then
\[
\begin{aligned}
\|\CN^1(\bv)_S\|_{\rL_s(\Omega)} 
&\leq C\big(
\|\bv_S\|_{\rL_{2s/(2-s)}(\Omega)}^{1-\theta}\|\nabla^2\bv_S\|_{\rL_q(\Omega)}^\theta
\|\nabla\bv_S\|_{\rL_{4s/(4-s)}(\Omega)}
\\
&\qquad\quad
+ (\|\pd_t\bv_\perp\|_{\rL_p(\BT, \rL_q(\Omega))}+ \|\bv_\perp\|_{\rL_p(\BT, \rH^2_q(\Omega))})
\|\bv_\perp\|_{\rL_p(\BT, \rH^1_s(\Omega))}
\big), 
\\
\|\CN^1(\bv)_S\|_{\rL_q(\Omega)}
&\leq C \big(
(\|\nabla^2\bv_S\|_{\rL_{s}(\Omega)}+\|\nabla^2\bv_S\|_{\rL_q(\Omega)})^2
\\
&\qquad\quad
+ (\|\pd_t\bv_\perp\|_{\rL_p(\BT, \rL_q(\Omega))}+ \|\bv_\perp\|_{\rL_p(\BT, \rH^2_q(\Omega))})^2
\big),
\\
\|\CN^1(\bv)_\perp\|_{\rL_p(\BT;\rL_s(\Omega))}
&\leq C \big(
\|\nabla^2\bv_S\|_{\rL_{s}(\Omega)}
+\|\pd_t\bv_\perp\|_{\rL_p(\BT, \rL_q(\Omega))}+ \|\bv_\perp\|_{\rL_p(\BT, \rH^2_q(\Omega))}
\big)
\\
&\qquad\quad
\times\big(
\|\bv_\perp\|_{\rL_p(\BT, \rH^1_s(\Omega))}+\|\bv_\perp\|_{\rL_p(\BT, \rH^1_q(\Omega))}
\big),
\\
\|\CN^1(\bv)_\perp\|_{\rL_p(\BT;\rL_q(\Omega))}
&\leq C \big(
\|\nabla^2\bv_S\|_{\rL_{s}(\Omega)}+\|\nabla^2\bv_S\|_{\rL_q(\Omega)}
+\|\pd_t\bv_\perp\|_{\rL_p(\BT, \rL_q(\Omega))}+ \|\bv_\perp\|_{\rL_p(\BT, \rH^2_q(\Omega))}
\big)
\\
&\qquad\quad
\times\|\bv_\perp\|_{\rL_p(\BT, \rH^1_q(\Omega))},
\end{aligned}
\]
where $\theta=(1/s-3/4)/(1/s+1/6-1/q)$.
\end{lem}

\begin{proof} 
Using the Gagliardo--Nirenberg inequality in exterior domains (see~\cite{CrispoMaremonti_GNI_2004})
and Young's inequality,
we obtain
\[
\begin{aligned}
\|\bv_S\|_{\rL_4(\Omega)}
&\leq C \|\bv_S\|_{\rL_{2s/(2-s)}(\Omega)}^{1-\theta}\|\nabla^2\bv_S\|_{\rL_q(\Omega)}^\theta,
\\
\|\bv_S\|_{\rL_\infty(\Omega)}
&\leq C \|\bv_S\|_{\rL_{3s/(3-2s)}(\Omega)}^{1-\theta_1}\|\nabla^2\bv_S\|_{\rL_q(\Omega)}^{\theta_1}
\leq C\big( \|\nabla^2\bv_S\|_{\rL_{s}(\Omega)}+\|\nabla^2\bv_S\|_{\rL_q(\Omega)}\big),
\\
\|\nabla\bv_S\|_{\rL_q(\Omega)}
&\leq C \|\nabla\bv_S\|_{\rL_{3s/(3-s)}(\Omega)}^{1-\theta_2}\|\nabla^2\bv_S\|_{\rL_q(\Omega)}^{\theta_2}
\leq C \big(\|\nabla^2\bv_S\|_{\rL_{s}(\Omega)}+\|\nabla^2\bv_S\|_{\rL_q(\Omega)}\big),
\end{aligned}
\]
with $\theta$ as above, $\theta_1=(1/s-2/3)/(1/s-1/q)$ and $\theta_2=(1/s-1/q-1/3)/(1/s-1/q)$. 
Combining these estimates 
with the interpolation inequality~\eqref{est:interpolation}
for $\bw=\bv_\perp$ and $D=\Omega$,
we obtain the asserted estimates of $\CN(\bv)_S$ directly from H\"older's inequality.
For the estimates of $\CN(\bv)_\perp$,
the H\"older inequality yields
\[
\begin{aligned}
\|\CN^1(\bv)_\perp\|_{\rL_p(\BT;\rL_s(\Omega))}
&\leq \|\bv_S\|_{\rL_{3s/(3-2s)}(\Omega)}\|\nabla\bv_\perp\|_{\rL_p(\BT;\rL_{3/2}(\Omega))}
+\|\bv_\perp\|_{\rL_p(\BT;\rL_3(\Omega))}\|\nabla\bv_S\|_{\rL_{3s/(3-s)}(\Omega)}
\\
&\ \
+ \|\bv_\perp\|_{\rL_\infty(\BT;\rL_\infty(\Omega))}\|\nabla\bv_\perp\|_{\rL_p(\BT;\rL_s(\Omega))}
+ \|\bv_\perp\|_{\rL_\infty(\BT;\rL_\infty(\Omega))}\|\nabla\bv_\perp\|_{\rL_1(\BT;\rL_s(\Omega))},
\\
\|\CN^1(\bv)_\perp\|_{\rL_p(\BT;\rL_q(\Omega))}
&\leq \|\bv_S\|_{\rL_{\infty}(\Omega)}\|\nabla\bv_\perp\|_{\rL_p(\BT;\rL_{q}(\Omega))}
+\|\bv_\perp\|_{\rL_p(\BT;\rL_\infty(\Omega))}\|\nabla\bv_S\|_{\rL_{q}(\Omega)}
\\
&\ \
+ \|\bv_\perp\|_{\rL_\infty(\BT;\rL_\infty(\Omega))}\|\nabla\bv_\perp\|_{\rL_p(\BT;\rL_q(\Omega))}
+ \|\bv_\perp\|_{\rL_\infty(\BT;\rL_\infty(\Omega))}\|\nabla\bv_\perp\|_{\rL_1(\BT;\rL_q(\Omega))}.
\end{aligned}
\]
Since $s\leq 4/3<3/2<3<q$, 
the remaining estimates now follow by Sobolev embeddings and interpolation as before.
\end{proof}

After these preparations, we prove the main theorems of this article. 
We begin with the case $\kappa=0$.

\begin{proof}[Proof of Theorem~\ref{mainthm:Stokes}]
Consider $\CI_{0,\rho}$ as in \eqref{eq:fpset}
with
\[
\|(\bv, \fq)\|_{\CI_0} 
:= \|\pd_t\bv\|_{\rL_p(\BT, \rL_q(\Omega))}
+ \|\bv\|_{\rL_p(\BT, \rH^2_q(\Omega))}
+ \|\nabla\fq\|_{\rL_p(\BT, \rL_q(\Omega))}
+ <\bv>_{p,1} + <\nabla\bv>_{p, 2}. 
\]
Let $(\bv, \fq)\in\CI_{0,\rho}$.
In virtue of Lemma~\ref{lem:rhs.loc} and Lemma~\ref{lem:nonlinest.Stokes},
Theorem~\ref{thm:tpStokes} implies
the existence of a solution
$(\bu,\fp)$ to~\eqref{eq:fpsystem}
such that
\[
\begin{aligned}
\|(\bu, \fp)\|_{\CI_0}
&\leq C\big(
<\bff_S>_3 + <\bF_S>_2 + <\bff_\perp>_{p, 2} + <\bF_\perp>_{p, 1} 
\\
&\qquad\quad
+<\CN^1(\bv)_S>_3 + <\tilde\CN^1(\bv)_S>_2
+<\CN^1(\bv)_\perp>_{p, 2} 
+ <\tilde\CN^1(\bv)_\perp>_{p, 1} 
\\
&\qquad\quad
+ \|\CN^2(\bv)\|_{\rL_p(\BT, \rL_q(\Omega))}
+ \|\CL(\bv,\fq)\|_{\rL_p(\BT, \rL_q(\Omega))}
+\|\bh\|_{\rT_{p,q}(\Gamma\times\BT)}
\big)
\\
&\leq C
\big(\varepsilon^2 + \|(\bv, \fq)\|_{\CI_0}^2 + \varepsilon_0\|(\bv, \fq)\|_{\CI_0}^2 + 
\varepsilon_0 \|(\bv, \fq)\|_{\CI_0} \big) 
\\
&\leq C(\varepsilon^2\rho^{-1} + \rho + \varepsilon_0\rho + \varepsilon_0)\rho,
\end{aligned}
\]
where $C>0$ does not depend on the choice of $(\bv, \fq)$.
Choosing $\rho=\varepsilon$ and $\varepsilon,\varepsilon_0>0$ sufficiently small,
we have $C(\varepsilon^2\rho^{-1} + \rho + \varepsilon_0\rho + \varepsilon_0)<1$,
so that the solution map
$\Xi_0\colon(\bv,\fq)\mapsto(\bu,\fp)$
is a well-defined self-mapping on $\CI_{0,\rho}$.

Since $\CL(\bv,\fq)$ is linear and $\CN(\bv)$ is quadratic
in $(\bv,\fq)$,
similar arguments lead to a constant $C>0$ such that
\[
\begin{aligned}
&\|\Xi_0(\bv_1, \fq_1)-\Xi_0(\bv_2, \fq_2)\|_{\CI_0}
\\
&\quad\leq C\big(
<\CN^1(\bv_1)_S-\CN^1(\bv_2)_S>_3 + <\tilde\CN^1(\bv_1)_S-\tilde\CN^1(\bv_2)_S>_2
+<\CN^1(\bv_1)_\perp-\CN^1(\bv_2)_\perp>_{p, 2} 
\\
&\qquad\quad
+ <\tilde\CN^1(\bv_1)_\perp-\tilde\CN^1(\bv_2)_\perp>_{p, 1} 
+ \|\CN^2(\bv_1)-\CN^2(\bv_2)\|_{\rL_p(\BT, \rL_q(\Omega))}
\\
&\qquad\quad
+ \|\CL(\bv_1,\fq_1)-\CL(\bv_2,\fq_2)\|_{\rL_p(\BT, \rL_q(\Omega))}\big)
\\
&\quad
\leq C  \big(
\|(\bv_1, \fq_1)\|_{\CI_0}
+ \|(\bv_2, \fq_2)\|_{\CI_0}
+\varepsilon_0\|(\bv_1, \fq_1)\|_{\CI_0} 
+\varepsilon_0\|(\bv_2, \fq_2)\|_{\CI_0} 
+\varepsilon_0 ) \|(\bv_1+\bv_2, \fq_1-\fq_2)\|_{\CI_0}
\\
&\quad
\leq C(\rho + \varepsilon_0\rho + \varepsilon_0)
\|(\bv_1-\bv_2, \fq_1-\fq_2)\|_{\CI_0},
\end{aligned}
\]
for all $(\bv_1, \fq_1),\,(\bv_2, \fq_2)\in\CI_{0,\rho}$.
Again we have $\rho=\varepsilon$,
and choosing $\varepsilon,\varepsilon_0>0$ so small that
$C(\varepsilon + \varepsilon_0\varepsilon + \varepsilon_0)<1$,
we see that $\Xi_0$ is also a contraction.
Therefore, the contraction mapping principle
yields existence of a fixed-point of $\Xi_0$,
which is a solution to \eqref{eq:systemref}
with the asserted properties.
\end{proof}

We treat the case $\kappa\neq 0$ in a similar way, starting with the framework 
of functions with anisotropic pointwise decay.

\begin{proof}[Proof of Theorem~\ref{mainthm:Oseen}]
Consider $\CI_{\kappa,\rho}$ as in \eqref{eq:fpset}
with
\[
\begin{aligned}
\|(\bv, \fq)\|_{\CI_{\kappa}} 
&:= \|\pd_t\bv\|_{\rL_p(\BT, \rL_q(\Omega))}
+ \|\bv\|_{\rL_p(\BT, \rH^2_q(\Omega))}
+ \|\nabla\fq\|_{\rL_p(\BT, \rL_q(\Omega))}
\\
&\qquad
+\,|\kappa|^\delta<\bv_S>^w_{1, \delta} 
+ \,|\kappa|^\delta<\nabla \bv_S>^w_{3/2, 1/2+\delta}
+ <\bv_\perp>_{p,1+\delta} + <\nabla\bv_\perp>_{p, 2+\delta}.
\end{aligned}
\]
Let $(\bv, \fq)\in\CI_{\kappa,\rho}$.
In virtue of Lemma~\ref{lem:rhs.loc} and Lemma~\ref{lem:nonlinest.Oseen},
Theorem~\ref{thm:tpOseen} implies
the existence of a solution
$(\bu,\fp)$ to~\eqref{eq:fpsystem}
such that
\[
\begin{aligned}
\|(\bu, \fp)\|_{\CI_\kappa}
&\leq C\big(
<\bff_S>^w_{5/2,1/2+2\delta}
+ <\bff_\perp>_{p, 2+\delta} +<\bF_\perp>_{p,1+\delta}
+<\CN^1(\bv)_S>^w_{5/2, 1/2+2\delta}
\\
&\qquad\quad
{+} <\CN^1(\bv)_\perp>_{p, \ell+1} 
+ <\tilde\CN^1(\bv)_\perp>_{p, \ell} 
+ \|\CN^2(\bv)\|_{\rL_p(\BT, \rL_q(\Omega))}
\\
&\qquad\quad
{+}\ \|\CL(\bv,\fq)\|_{\rL_p(\BT, \rL_q(\Omega))}
+\|\bh\|_{\rT_{p,q}(\Gamma\times\BT)}
\big)
\\
&\leq C
\big(\varepsilon^2|\kappa|^{2\delta} + (1+|\kappa|^{-2\delta})\|(\bv, \fq)\|_{\CI_\kappa}^2 + \varepsilon_0\|(\bv, \fq)\|_{\CI_\kappa}^2 + 
\varepsilon_0 \|(\bv, \fq)\|_{\CI_\kappa} \big) 
\\
&\leq C\big(\varepsilon^2|\kappa|^{2\delta}\rho^{-1} + (1+|\kappa|^{-2\delta})\rho 
+ \varepsilon_0\rho + \varepsilon_0\big)\rho.
\end{aligned}
\]
We further proceed similarly to the proof of Theorem~\ref{thm:tpOseen}.
We choose $\rho=\varepsilon|\kappa|^{2\delta}$.
Then 
the solution map
$\Xi_\kappa\colon(\bv,\fq)\mapsto(\bu,\fp)$
is a self-mapping on $\CI_{\kappa,\rho}$
if 
$C(\varepsilon + \varepsilon(|\kappa|^{2\delta}+1) 
+ \varepsilon_0(1+\varepsilon|\kappa|^{2\delta}))<1
$.
This is the case for all $|\kappa|\leq\kappa_0$
if we choose $\varepsilon,\varepsilon_0>0$ sufficiently small.
Arguing as in the previous proof, 
we further obtain
\[
\begin{aligned}
&\|\Xi_\kappa(\bv_1, \fq_1)-\Xi_\kappa(\bv_2, \fq_2)\|_{\CI_\kappa}
\\
&\quad\leq C\big(
(1+|\kappa|^{-2\delta})(\|(\bv_1, \fq_1)\|_{\CI_\kappa}+\|(\bv_2, \fq_2)\|_{\CI_\kappa})
+\varepsilon_0(1+\|(\bv_1, \fq_1)\|_{\CI_\kappa}+\|(\bv_2, \fq_2)\|_{\CI_\kappa})
\big)
\\
&\qquad\qquad\times
\|(\bv_1-\bv_2, \fq_1-\fq_2)\|_{\CI_\kappa}
\\
&\quad\leq C\big(
\rho
+|\kappa|^{-2\delta}\rho
+ \varepsilon_0+\varepsilon_0\rho \big)\|(\bv_1-\bv_2, \fq_1-\fq_2)\|_{\CI_\kappa}
\end{aligned}
\]
for $(\bv_1, \fq_1),(\bv_2, \fq_2)\in\CI_{\kappa,\rho}$.
With the choice $\rho=\varepsilon|\kappa|^{2\delta}$
we see that $\Xi_\kappa$ is also contractive for all $|\kappa|\leq \kappa_0$
for $\varepsilon,\varepsilon_0>0$ sufficiently small.
Finally, Banach's fixed-point theorem yields an element
$(\bw,\fq)\in\CI_\kappa$ with
$(\bw,\fq)=\Xi_\kappa(\bw,\fq)$,
which completes the proof.
\end{proof}

Finally, we treat the case $\kappa\neq0$ in a framework of homogeneous Sobolev spaces.

\begin{proof}[Proof of Theorem~\ref{mainthm:Oseen.int}]
Consider $\CI_{\kappa,\rho}$ as in \eqref{eq:fpset}
with
\[
\begin{aligned}
\|(\bv, \fq)\|_{\CI_{\kappa}} 
&:= \|\nabla^2\bv_S\|_{\rL_s(\Omega)}
+ |\kappa|^{1/4}\|\nabla\bv_S\|_{\rL_{4s/(4-s)}(\Omega)}
+ |\kappa|^{1/2}\|\bv_S\|_{\rL_{2s/(2-s)}(\Omega)}
+ |\kappa|\,\|\partial_1\bv_S\|_{\rL_s(\Omega)}
\\
&\qquad
+\|\nabla^2\bv_S\|_{\rL_q(\Omega)}
+ \|\pd_t\bv_\perp\|_{\rL_p(\BT, \rL_s(\Omega))}
+ \|\bv_\perp\|_{\rL_p(\BT, \rH^2_s(\Omega))}
+ \|\nabla\fq\|_{\rL_p(\BT, \rL_s(\Omega))}
\\
&\qquad
+ \|\pd_t\bv_\perp\|_{\rL_p(\BT, \rL_q(\Omega))}
+ \|\bv_\perp\|_{\rL_p(\BT, \rH^2_q(\Omega))}
+ \|\nabla\fq\|_{\rL_p(\BT, \rL_q(\Omega))}.
\end{aligned}
\]
Let $(\bv, \fq)\in\CI_{\kappa,\rho}$.
In virtue of Lemma~\ref{lem:rhs.loc} and Lemma~\ref{lem:nonlinest.Oseen.int},
Theorem~\ref{thm:tpOseen.int} implies
the existence of a solution
$(\bu,\fp)$ to~\eqref{eq:fpsystem}
such that
\[
\begin{aligned}
\|(\bu, \fp)\|_{\CI_\kappa}
&\leq C\big(
\|\bff\|_{\rL_{p}(\BT, \rL_s(\Omega))}
+\|\bff\|_{\rL_{p}(\BT, \rL_q(\Omega))}
+\|\CN^1(\bv)\|_{\rL_{p}(\BT, \rL_s(\Omega))}
+\|\CN^1(\bv)\|_{\rL_{p}(\BT, \rL_q(\Omega))}
\\
&\qquad\qquad
+ \|\CN^2(\bv)\|_{\rL_p(\BT, \rL_s(\Omega))}
+ \|\CN^2(\bv)\|_{\rL_p(\BT, \rL_q(\Omega))}
+ \|\CL(\bv,\fq)\|_{\rL_p(\BT, \rL_s(\Omega))}
\\
&\qquad\qquad
+ \|\CL(\bv,\fq)\|_{\rL_p(\BT, \rL_q(\Omega))}
+\|\bh\|_{\rT_{p,s}(\Gamma\times\BT)}
+\|\bh\|_{\rT_{p,q}(\Gamma\times\BT)}
\big).
\end{aligned}
\]
Observe that $C$ is independent of $|\kappa|\leq 1$.
Since $\CN^2(\bv)$ and $\CL(\bv)$ vanish outside $\Omega_{2b}$,
we can use Lemma~\ref{lem:rhs.loc} to estimate
\[
\begin{aligned}
&\|\CN^2(\bv)\|_{\rL_p(\BT, \rL_s(\Omega))}
+ \|\CN^2(\bv)\|_{\rL_p(\BT, \rL_q(\Omega))}
+ \|\CL(\bv,\fq)\|_{\rL_p(\BT, \rL_s(\Omega))}
+ \|\CL(\bv,\fq)\|_{\rL_p(\BT, \rL_q(\Omega))}
\\
&\qquad
\leq C\big(
\|\CN^2(\bv)\|_{\rL_p(\BT, \rL_q(\Omega_{2b}))}
+ \|\CL(\bv,\fq)\|_{\rL_p(\BT, \rL_q(\Omega_{2b}))}
\big)
\leq C\varepsilon_0\big(
\|(\bv, \fq)\|_{\CI_\kappa}  +\|(\bv, \fq)\|_{\CI_\kappa}^2
\big),
\end{aligned}
\]
where we used 
\[
\|\bv_S\|_{\rH^2_q(\Omega_{2b})}
\leq C\big(
\|\bv_S\|_{\rL_{3s/(3-2s)}(\Omega_{2b})}+ \|\nabla^2\bv_S\|_{\rL_q(\Omega_{2b})}
\big)
\leq C\big(
\|\nabla^2\bv_S\|_{\rL_s(\Omega)}+ \|\nabla^2\bv_S\|_{\rL_q(\Omega)}
\big).
\]
We further split $\CN^1(\bv)$ 
into steady-state and oscillatory part and combine 
the estimates from Lemma~\ref{lem:nonlinest.Oseen.int} 
with the previous one to conclude
\[
\begin{aligned}
\|(\bu, \fp)\|_{\CI_\kappa}
&\leq C\big(
\varepsilon^2|\kappa|^{(1+\delta)/2}
+|\kappa|^{-(1-\theta)/2}\|(\bv, \fq)\|_{\CI_\kappa}^2
+\|(\bv, \fq)\|_{\CI_\kappa}^2
+\varepsilon_0(\|(\bv, \fq)\|_{\CI_\kappa}  +\|(\bv, \fq)\|_{\CI_\kappa}^2)
\big)
\\
&\leq C\big(
\varepsilon^2|\kappa|^{(1+\delta)/2}\rho^{-1}
+|\kappa|^{-(1-\theta)/2}\rho
+\rho
+ \varepsilon_0(1+\rho)
 \big)\rho
\end{aligned}
\]
with $\theta$ as in Lemma~\ref{lem:nonlinest.Oseen.int},
so that $\theta\in[0,3/10)$.
We now choose $\rho=\varepsilon |\kappa|^{1/2}$. 
Then the solution map
$\Xi_\kappa\colon(\bv,\fq)\mapsto(\bu,\fp)$
is a self-mapping on $\CI_{\kappa,\rho}$ 
if 
\[
C\big(
\varepsilon|\kappa|^{\delta/2}
+\varepsilon|\kappa|^{\theta/2}
+\varepsilon|\kappa|^{1/2}
+ \varepsilon_0(1+\varepsilon|\kappa|^{1/2})
 \big)<1.
\]
Since the constant $C$ is independent of $\kappa$,
we can take $\varepsilon,\varepsilon_0>0$
so small that this is satisfied
for all $|\kappa|\leq\kappa_0$.
Modifying the previous argument,
we can further show
\[
\begin{aligned}
&\|\Xi_\kappa(\bv_1, \fq_1)-\Xi_\kappa(\bv_2, \fq_2)\|_{\CI_\kappa}
\\
&\leq C\big(
(1+|\kappa|^{-(1-\theta)/2})(\|(\bv_1, \fq_1)\|_{\CI_\kappa}+\|(\bv_2, \fq_2)\|_{\CI_\kappa})
+\varepsilon_0(1+\|(\bv_1, \fq_1)\|_{\CI_\kappa}+\|(\bv_2, \fq_2)\|_{\CI_\kappa})
\big)
\\
&\qquad\qquad\times
\|(\bv_1-\bv_2, \fq_1-\fq_2)\|_{\CI_\kappa}
\\
&\leq C\big(
\rho
+|\kappa|^{-(1-\theta)/2}\rho
+ \varepsilon_0+\varepsilon_0\rho \big)\|(\bv_1-\bv_2, \fq_1-\fq_2)\|_{\CI_\kappa}
\end{aligned}
\]
for some $C>0$ independent of $\kappa$
and for all $(\bv_1, \fq_1),\,(\bv_2, \fq_2)\in\CI_{\kappa,\rho}$.
With $\rho=\varepsilon|\kappa|^{1/2}$, 
we can choose
and $\varepsilon,\varepsilon_0>0$ so small
that $\Xi_{\kappa}$ is a contractive self-mapping
for all $|\kappa|\leq\kappa_0$.
Now Banach's fixed-point theorem 
yields existence of a unique fixed point in $\CI_{\kappa,\rho}$,
which is a solution to \eqref{eq:systemref} as claimed.
\end{proof}


\end{document}